\documentclass[12pt,reqno,a4paper]{article}

\usepackage[outer=2.5cm,top=2.5cm,bottom=2.5cm,inner=2.0cm]{geometry}

\usepackage{xfrac}
\usepackage{hyperref}
\usepackage{amsmath}
\usepackage{amsthm}
\usepackage{amssymb}
\usepackage{amsfonts}
\usepackage{dsfont}
\usepackage{float}
\usepackage{amscd}
\usepackage{amsfonts}
\usepackage{stmaryrd}
\usepackage{commath}
\usepackage{mathrsfs}
\usepackage{enumitem}
\usepackage{graphicx}
\usepackage{subcaption}
\usepackage{tikz}
\usepackage{verbatim}

\usepackage{algorithm}
\usepackage{algorithmic}
\usepackage{pgfplots}
\usepackage[numbers]{natbib}
\usepackage{mathrsfs}

\usepackage{todonotes}
\usepackage{placeins}

\font\msbm=msbm10

\numberwithin{equation}{section}

\theoremstyle{plain}
\newtheorem{theorem}{Theorem}[section]

\newtheorem{corollary}[theorem]{Corollary}
\newtheorem{proposition}[theorem]{Proposition}

\theoremstyle{definition}

\newtheorem{remark}[theorem]{Remark}
\def\mathbb#1{\hbox{\msbm{#1}}}

\newcommand{\ltwo}{\ensuremath{L_{2}(\T^d)}}

\newcommand{\dd}{ {\rm d}  }

\newcommand{\field}[1]{\ensuremath{\mathds{#1}}}

\newcommand{\N}{\field{N}}

\newcommand{\Prob}{\field P}

\newcommand{\C}{\field{C}}
\newcommand{\tor}{\field T} 

\newcommand{\Z}{\field{Z}}
\newcommand{\n}{\ensuremath{{\N}_0}}
\newcommand{\R}{\field{R}}
\newcommand{\T}{\field{T}}

\newcommand{\Id}{\ensuremath{\mathrm{Id}}}

\newcommand{\mix}{{\rm mix}}

\newcommand{\bk}{\mathbf{k}}

\newcommand{\bx}{\mathbf{x}}
\newcommand{\by}{\mathbf{y}}

\newcommand{\bL}{\mathbf{L}}

\newcommand{\bX}{\mathbf{X}}

\newcommand{\be}{\begin{equation}}
\newcommand{\ee}{\end{equation}}
\newcommand{\beq}{\begin{eqnarray}}
\newcommand{\beqq}{\begin{eqnarray*}}
\newcommand{\eeq}{\end{eqnarray}}
\newcommand{\eeqq}{\end{eqnarray*}}

\newcommand{\trace}[1]{\operatorname{tr}(#1)}

\newcommand{\diag}{\mathrm{diag}}

\begin{document}

\title{A note on sampling recovery of multivariate functions in the uniform norm}
\author{Kateryna Pozharska, Tino Ullrich\footnote{Corresponding author: tino.ullrich@mathematik.tu-chemnitz.de} \\\\
Institute of Mathematics of NAS of Ukraine, 01024 Kyiv, Ukraine\\
TU Chemnitz, Faculty of Mathematics, 09107 Chemnitz, Germany}

\maketitle

\begin{abstract}
\small
We study the recovery of multivariate functions from reproducing kernel Hilbert spaces in the uniform norm.
Our main interest is to obtain preasymptotic estimates for the corresponding sampling numbers. We obtain results in terms of the decay of related singular numbers of the compact embedding into $L_2(D,\varrho_D)$ multiplied with the supremum of the Christoffel function of the subspace spanned by the first $m$ singular functions. Here the measure $\varrho_D$ is at our disposal. As an application we obtain near optimal upper bounds for the sampling numbers for periodic Sobolev type spaces with general smoothness weight. Those can be bounded in terms of the corresponding benchmark approximation number in the uniform norm, which allows for preasymptotic bounds. By applying a recently introduced sub-sampling technique related to Weaver's conjecture we mostly lose a $\sqrt{\log n}$ and sometimes even less. Finally we point out a relation to the corresponding Kolmogorov numbers.
\\


\medskip

\small%

\end{abstract}

\section{Introduction}

We consider the problem of the recovery of complex-valued multivariate functions on a compact domain $D \subset \R^d$ in the uniform norm. We are allowed to use function samples on a suitable set of nodes $\bX :=\{\bx^1,...,\bx^n\} \subset D$. The functions are modeled as elements from some reproducing kernel Hilbert space $H(K)$ with kernel $K\colon D\times D\to \C$. The nodes $\bX$ are drawn independently according to a tailored probability measure depending on the spectral properties of the embedding of $H(K)$ in $L_2(D,\varrho_D)$. This ``random information'' is used for the whole class of functions and our main focus in this paper is on worst-case errors. We additionally apply a recently introduced sub-sampling technique, see \cite{Ni_Sc_Ut20} to further reduce the sampling budget. Note that in one scenario the measure $\varrho_D$ is at our disposal and represents a certain degree of freedom. In another scenario the measure $\varrho_D$ is a fixed measure, namely the one where the data is coming from.

Authors distinguish two main paradigms for the recovery of functions from ``samples''. Here ``samples'' refer to the available information. The first one uses general ``linear'' information about functions, i.e., evaluations given by $n$ arbitrary linear continuous functionals as a number of inner products of the Hilbert space,  e.g., Fourier or wavelet coefficients. The corresponding quantity that measures the minimal worst case error is given by the Gelfand numbers/widths, see the monographs \cite{NoWoI,NoWoII,NoWoIII}. In case of a Hilbert source space (like in our setting) these quantities equal the so-called approximation numbers which are usually denoted by $a_n$ and represent the worst-case error obtained by linear recovery algorithms. Precisely, let $H(K)$ be a reproducing kernel Hilbert space (RKHS) of multivariate complex-valued functions $f(\bx)= f(x_1, \dots, x_d)$, $\bx \in D\subset \R$  with the kernel
 $K\colon D\times D \rightarrow \C$. Consider the identity operator $\Id\colon H(K) \rightarrow F $. Further, by $\mathcal{L}(H(K), F)$ we denote the set of linear continuous operators $A\colon H(K) \rightarrow F$. The approximation numbers $a_n ( \Id)$ are defined as follows
\be\label{def_approxim_numbers}
a_n \left( \Id \colon H(K) \rightarrow F \right):=\inf\limits_{ \substack{ A \in \mathcal{L}(H(K), F) \\ {\rm rank} A<n }} \sup\limits_{ \| f \|_{H(K)} \leq 1} \|  f - Af  \|_F.
\ee
If $F$ is a Hilbert space, then $a_n$ represents the singular numbers $\sigma_n$ of the corresponding embedding.

Several practical applications are modeled with ``standard information'', i.e., a finite number of function values at some points.
Let $f(\boldsymbol{\rm x}^1), \dots, f(\boldsymbol{\rm x}^n)$ be the given data, where
$\boldsymbol{\rm X} = \{ \boldsymbol{\rm x}^1, \dots,  \boldsymbol{\rm x}^n \}$, $\boldsymbol{\rm x}^i\subset D$, $i=1,\dots,n $,
are the nodes.
The corresponding sampling numbers are defined as follows:
\be\label{def_sampling_numbers}
g_{n} \left( \Id\colon H(K) \rightarrow F \right) :=
\inf_{\boldsymbol{\rm X}= \{\boldsymbol{\rm x}^1, \dots,  \boldsymbol{\rm x}^n \}}
\inf_{ R \in \mathcal{L}(\C^n, F) }
\sup_{\|f \|_{H(K)} \leq 1}  \| f - R \left( f( \boldsymbol{\rm X} )\right) \|_{F}.
\ee
Note that we always have $a_n ( \Id) \leq g_n ( \Id)$.

Our goal is to recover functions using a (weighted) least squares algorithm and measure the error in $\ell_\infty$. Function values are taken at independently drawn random nodes in $D$ according to a tailored probability measure $\varrho$. The main feature of this approach is that the nodes are drawn once for the whole class. This is usually termed as ``random information''. Note that we do not investigate Monte Carlo algorithms in this paper.

A recent breakthrough was made by D.~Krieg and M.~Ullrich~\cite{KrUl19}
in establishing a new connection between sampling numbers and singular values for $L_2$-approximation. This has been extended by L. K\"{a}mmerer, T. Ullrich and T. Volkmer~\cite{KUV19} and bounds with high probability have been obtained, see also M.~Moeller and T.~Ullrich~\cite{MoUll20} and M.Ullrich \cite{Ul20} for further refinements. In all the mentioned contributions we need a logarithmic oversampling, i.e, $n \asymp m\log m$. A new sub-sampling technique has been applied by N. Nagel, M. Sch\"afer and T. Ullrich in \cite{Ni_Sc_Ut20}, where it is shown that $n = \mathcal{O}(m)$ samples are sufficient. For non-Hilbert function classes we refer to
D.~Krieg and M.~Ullrich~\cite{KrUl20} and V.N.~Temlyakov~\cite{Teml2020}, where a connection to Kolmogorov widths is pointed out. There is also a relation in the $\ell_\infty$-setting, see
\cite{KaTi2021} for details.

In the present paper we obtain  worst case recovery guarantees with high probability for the weighted least squares recovery operator $\widetilde{S}_{\bX}^m$
(see Algorithm~\ref{algo1:reweighted}) where the error is measured in the supremum norm of the space $\ell_{\infty}(D)$ of bounded on $D$ functions. The operator $\widetilde{S}_{\bX}^m$ uses a density function $\varrho_m$ tailored to the spectral properties of the corresponding embedding into $L_2(D,\varrho_D)$, see \eqref{density_f} and \eqref{sd}.
We need to impose stronger conditions, namely we consider a compact set $D \subset \R^d$, a continuous kernel $K(\bx, \by)$ such that $\sup_{\bx \in D} \sqrt{K(\bx,\bx)} < \infty$ and a Borel measure $\varrho_D$ with full support
(such that Mercer's theorem is applicable). Then we get for $m:= \lfloor\frac{n}{c_1r\log n}\rfloor$, $r>1$, the following error bound with probability larger than $1-c_2n^{1-r}$
\begin{equation}\label{f1}
\sup\limits_{\|f\|_{H(K)}\leq 1} \| f- \widetilde{S}^{m}_{\boldsymbol{\rm X}}f\|^2_{\ell_{\infty}(D)}
\leq c_3 \max\Big\{\frac{N_{K,\varrho_D}(m)}{m}\sum\limits_{k\geq \lfloor m/2 \rfloor}\sigma_k^2,
\sum\limits_{k\geq \lfloor m/2 \rfloor}\frac{N_{K,\varrho_D}(4k)\sigma_k^2}{k}\Big\}
\end{equation}
with three universal constants $c_1, c_2, c_3 >0$, where $N_{K,\varrho_D}(m)$ denotes the supremum over the spectral function (Christoffel function) which is defined in Section~\ref{section_RKHS}. When applying this theorem to periodic spaces with mixed smoothness $H^s_{\mix}(\tor^d)$ we obtain the following result (see Theorem
\ref{preasymp_Hmix_sharp}).

If $s>(1+\log_2 d)/2$, $r>1$, $m=\lfloor n/(10r\log n)
\rfloor$ and $\beta = 2s/(1+\log_2 d)>1$ then
\begin{equation}\label{f1b}
	\sup_{\|f   \|_{H^{s}_{\mix}(\T^d)} \leq 1}  \| f - S_{\bX}^m f \|^2_{L_\infty(\T^d)}
\leq 1612 \left( \frac{16}{3} \right)^{\beta}  \frac{\beta}{\beta-1} \left( \frac{m}{2} -1 \right)^{-\beta +1}
\end{equation}
with probability larger than $1-3n^{1-r}$\,. Clearly, we have to determine the norm in $H^s_{\mix}(\T^d)$. In the above statement we used the $\#$-norm, see \eqref{periodic_H_sharp_norm}. This is a preasymptotic statement.

Note, that the sampling operator uses $n \asymp m\log m$ many sample points such that we obtain a bound for $g_{\lfloor cm\log m \rfloor}$. By a recently introduced sub-sampling technique (see \cite{Ni_Sc_Ut20}) we obtain for the sampling numbers another result which may be better in many cases, namely
\begin{equation}\label{f2}
 g_{n}( \Id\colon H(K) \rightarrow \ell_{\infty}(D))^2 \leq
 c_4\max\Big\{\frac{N_{K,\varrho_D}(n)\log n}{n}\sum\limits_{k\geq \lfloor c_5n \rfloor}\sigma_k^2,
\sum\limits_{k\geq \lfloor c_5n \rfloor}\frac{N_{K,\varrho_D}(4k)\sigma_k^2}{k}\Big\}\,,
\end{equation}
where $c_4, c_5>0$ are universal constants. We evaluate this bound for several relevant examples to obtain new bounds for sampling numbers. Here the question for reasonable assumptions arises to ensure the polynomial order $q^*$ of the error decay, i.e., the supremum over all $q$ such that the worst case error is of the order $n^{-q}$. We improve a result by F.~Y.~Kuo, G.~W.~Wasilkowski and H.~Wo\'{z}niakowski~\cite{KWW09} and show that $q^{\rm std}_\infty = p-1/2$ in case $N_{K,\varrho_D}(k) = \mathcal{O}(k)$ and $\sigma_n = \mathcal{O}(n^{-p})$. We actually prove more: If there is a measure $\varrho_D$ such that $N_{K,\varrho_D}(k) = \mathcal{O}(k)$ we obtain as a consequence of \eqref{f1} and \eqref{f2} together with \cite[Lem.\ 3.3]{CoKuSi16} the relation
\begin{equation}\label{main}
	g_n(\Id) \leq C_{\varrho_D,K}\min\{a_{\lfloor n/(c_6 \log n)\rfloor}(\Id),\sqrt{\log n}\cdot a_{\lfloor c_5 n \rfloor}(\Id)\}\,,
\end{equation}
where $c_6>0$ is an absolute constant and $C_{\varrho_D,K}>0$ depends on the kernel $K$ and $\varrho_D$, see Corollary \ref{cor2}. This improves on a recent result by L.~K\"{a}mmerer, see \cite[Theorem~4.11]{Kam19}, which was stated under stronger conditions (that is, in turn, an improvement of the previous result by L.~K\"{a}mmerer and T.~Volkmer~\cite{Kam_Vol19})). It shows, that in case $N_{K,\varrho_D}(k) = \mathcal{O}(k)$ we only loose a $\sqrt{\log n}$ or sometimes even less (assuming a polynomial decay of the $a_n$). In case that $N_{K,\varrho_D}(k) = \mathcal{O}(k^u)$, with $u>1$, we have at least that $q^{\rm std}_\infty \geq p-u/2$.

As to spaces with mixed smoothness $H^s_{\mix}(\T^d)$, $s>1/2$, we know by the work of V.N. Temlyakov \cite{Te93} that
$$
	g_n(\operatorname{I}_s) \asymp_{s, d} n^{-s+1/2}(\log n)^{(d-1)s}\,,
$$
where the optimal algorithm is a (deterministic) Smolyak type interpolation operator.
Our algorithm yields the bound $$
	g_n(\operatorname{I}_s:H^s_{\mix}(\T^d)\to L_\infty(\T^d)) \lesssim_{s, d} n^{-s+1/2}(\log n)^{(d-1)s+\min\{s-1/2,1/2\}}\,,
$$
which is only worse by a small $d$-independent $\log$-term. However, as indicated in \eqref{f1b} we also obtain preasymptotic bounds for $g_n(\operatorname{I}_s)$ from \eqref{f1} and \eqref{f2}. The open question remains whether random iid nodes are asymptotically optimal for this framework. Such a question has been recently answered positively for isotropic Sobolev spaces in \cite{Kr20}.

Let us point out that the above general results are also applicable to the more general periodic Sobolev type spaces $H^w(\T^d)$ with a weight sequence $(w(\bk))_{\bk \in \Z^d}$. If $(\tau_n)_n$ represents the non-increasing rearrangement of the square summable weight sequence $(1/w(\bk))_{\bk \in \Z^d}$ then we have
$$
  g_n(\operatorname{I}_w:H^w(\T^d) \to L_\infty(\T^d))^2 \leq c_7\min\Big\{\log n\sum\limits_{k\geq \lfloor c_9n \rfloor}\tau_k^2 , \sum\limits_{k\geq \lfloor n/(c_8 \log n) \rfloor} \tau_k^2\Big\}
$$
with three universal constants $c_7, c_8, c_9 >0$. In most of the cases we do not know whether there is a deterministic method which may compete with this bound.

{\bf Notation. }  By $\N$ we denote the set of natural numbers putting $\n = \N \cup \{0\}$, by $\Z$ - the set of integer, $\R$ - real, and $\C$ - complex numbers. As usual, $a\in \R$ is decomposed into $a=\lfloor a \rfloor + \{a\}$, where $0\leq \{a\} <1$ and $\lfloor a \rfloor \in \Z$. Let further $\N^d$, $\n^d$, $\Z^d$, $\R^d$, $\C^d$, $d \geq 2$, be the spaces of $d$-dimensional vectors $\bx = (x_1, \dots,x_d)$ with coordinates, respectively, from $\N$, $\n$, $\Z$, $\R$, and $\C$. The notation $\bx \leq \by$, where $\bx, \by \in \C^d$, stands for $x_i \leq y_i$, $i=1,\dots, d$ (relations $\geq$, $<$, $>$ and $=$ are defined analogically), $\overline{\bx}$ is complex conjugate to $\bx$, $\bx^{\top}$ is transpose to $\bx$. For $\bk\in \Z^d$ and $\bx\in \C^d$ we write $\bx^{\bk}:= \prod_{i=1}^d x_i^{k_i}$ assuming that $0^0:= 1$. Let also
$|\bk|= (|k_1|, \dots, |k_d|)$, $|\bk|_0$ be the number of nonzero components (sparsity) of $\bk$. The symbols $(\bx, \by)$ or $\bx\cdot \by$ mean the Euclidean scalar product in $\R^d$ or $\C^d$.
We use the notation $\log$ for the natural logarithm. By $\C^{m\times n}$ we denote the set of all $m\times n$ matrices $ \boldsymbol{\rm L}$ with complex entries, where, in what follows, $\| \boldsymbol{\rm L} \|$ or
$\| \boldsymbol{\rm L} \|_{2 \to 2}$ is their spectral norm (the largest singular value), $\boldsymbol{\rm L}^*$ is adjoint matrix to $ \boldsymbol{\rm L}$, i.e., the matrix obtained by taking transpose and then taking complex conjugate of each entry of $\boldsymbol{{\rm L}}$. For an operator $A\in \mathcal{L}(X, Y)$ between two Banach spaces we use the symbol $\| A \|$ for its operator norm, i.e.,
$$
\| A \|:= \sup\limits_{f\in X, f\neq 0} \frac{\|Af\|_Y}{\|f\|_X}.
$$
If $X, Y$ are Hilbert spaces, by $A^*$ we denote adjoint to $A$ operator $A^* \colon Y \rightarrow X$, such that for $f \in X$, $g \in Y$ the following equality for the corresponding scalar products holds: $(A f, g )_Y = ( f, A^* g)_X$. The notation $A^{-1}$ stands for inverse operator. For two sequences $(a_n)_{n=1}^{\infty},(b_n)_{n=1}^{\infty}\subset \R$  we write $a_n \lesssim b_n$ if there exists a constant $c>0$,
such that $a_n \leq cb_n$ for all $n$. We will write $a_n \asymp b_n$ if $a_n \lesssim b_n$ and $b_n \lesssim a_n$.
If the constant $c$ depends on the dimension $d$ and smoothness $s$, we indicate it by $\lesssim_{s,d}$ and $\asymp_{s,d}$.

\section{Reproducing kernel Hilbert spaces}\label{section_RKHS}
We will work in the framework of reproducing kernel Hilbert spaces. The relevant theoretical background can be found in \cite[Chapt.\ 1]{BeTh04} and \cite[Chapt.\ 4]{StChr08}.

Let $D \subset \R^d$ be an arbitrary subset and $\varrho_D$ be a measure on $D$. By $L_2(D,\varrho_D)$ denote
 the space of complex-valued square-integrable functions with respect to~$\varrho_D$, where
$$
(f,g)_{L_{2}(D, \varrho_D)} := \int_D f(\bx) \overline{g(\bx)} \dd {\varrho_{D}} (\bx),
$$
and, respectively,
$$
\| f \|_{L_{2}(D, \varrho_D)} = \left( \int_D |f(\bx)|^2 \dd {\varrho_D} (\bx) \right)^{1/2}.
$$
With $\ell_{\infty}(D)$ we denote the space of bounded on $D$ functions, for which
$$
\| f \|_{\ell_{\infty}(D)} := \sup \limits_{\bx \in D} |f(\bx)|.
$$
Note, that one do not need the measure $\varrho_D$ for this embedding. In fact, here we mean ``boundedness'' in the strong sense (in contrast to essential boundedness w.r.t. the measure $\varrho_D$).

We further consider a reproducing kernel Hilbert space $H(K)$ with a Hermitian positive definite kernel $K(\bx,\by)$ on $D \times D$. The crucial property is the identity
$$
		f(\bx) = ( f, K(\cdot,\bx) )_{H(K)}
$$
for all $\bx \in D$. It ensures that point evaluations are continuous functionals on $H(K)$. We will use the notation from \cite[Chapt.\ 4]{StChr08}. In the framework of this paper, the boundedness of the kernel $K$ is assumed, i.e.,
\begin{equation} \label{CK000}
		\|K\|_{\infty} := \sup\limits_{\bx \in D} \sqrt{K(\bx,\bx)} < \infty.
\end{equation}
 The condition (\ref{CK000}) implies that $H(K)$ is continuously embedded into $\ell_\infty(D)$, i.e.,
\begin{equation}\label{emb0}
		\|f\|_{\ell_{\infty}(D)} \leq  \|K\|_{\infty}\cdot\|f\|_{H(K)}\,.
\end{equation}
 The embedding operator
\begin{equation}\label{f00b}
\Id\colon H(K) \to L_2(D,\varrho_D)
\end{equation}
is Hilbert-Schmidt under the finite trace condition
 of the kernel
\begin{equation}\label{integrab}
	\trace{K}:=\|K\|^2_{2} = \int_{D} K(\bx,\bx)\dd\varrho_D(\bx) < \infty,
\end{equation}
 see \cite{HeBo04}, \cite[Lemma 2.3]{StSc12}, which is less strong than (\ref{CK000}).
 In view of considering in this paper a recovery of functions in the uniform norm, we assume that the kernel $K$ is bounded from now.
 We additionally assume that $H(K)$ is at least infinite dimensional.

Consider the embeddings $\Id$ from (\ref{f00b}) and $W_{\varrho_D} = \Id^* \circ \Id \colon H(K) \rightarrow  H(K)$.
Due to the compactness of $\Id$, the operator $W_{\varrho_D}$ provides an (at most) countable system
of strictly positive eigenvalues. Let $( \lambda_n)_{n=1}^{\infty}$ be a non-increasing rearrangement of these eigenvalues,
$( \sigma_n)_{n=1}^{\infty}$ be the sequence of singular numbers of $\Id$, i.e., $\sigma_k := \sqrt{\lambda_k}$, $k=1, 2, \dots$.
Further, with $(e^*_n (\bx))_{n=1}^{\infty} \subset H(K)$ denote the system of right eigenfunctions (that correspond to the ordered system $( \lambda_n)_{n=1}^{\infty}$), and
with $( \eta_n (\bx) )_{n=1}^{\infty} \subset L_2(D,\varrho_D)$ the system consisting of $\eta_k (\bx) := \sigma_k^{-1} e^*_k(\bx)$, $k=1, 2, \dots$,
which both represent orthonormal systems (ONS) in the respective spaces.

 We would like to emphasize that the embedding \eqref{f00b} is not necessarily injective.
  As discussed in \cite[p.\,127]{StChr08}, $\Id$ is in general {\em not} injective as it maps a function to an equivalence class. As a consequence, the system of eigenvectors $(e^*_n (\bx))_{n=1}^{\infty}$ may not be a basis in $H(K)$ (note that $H(K)$ may not even be separable). However, there are conditions which ensure that ONS $(e^*_n (\bx))_{n=1}^{\infty}$ is an orthonormal basis (ONB) in $H(K)$, see \cite[Chapt.~4.5]{StChr08}, which is related to Mercer's theorem, see \cite[Theorem\ 4.49]{StChr08}. Indeed, if we additionally assume that the kernel $K$ is bounded and continuous on $D \times D$ (for a compact domain $D \subset \R^d$), then $H(K)$ is separable and consists of continuous functions, see \cite[Theorems\ 16, 17]{BeTh04}. If we finally assume that the measure $\varrho_D$ is a Borel measure with full support then $(e^*_n (\bx))_{n=1}^{\infty}$ is a complete orthonormal system, i.e., ONB in $H(K)$.

In this case we have the pointwise identity (see, e.g., \cite[Cor.~4]{BeTh04})
\begin{equation}\label{ptwise}
	K(\bx,\by)=\sum\limits_{k=1}^{\infty} \overline{e^*_k(\by)}e^*_k(\bx), \quad \quad \bx,\by \in D\,.
\end{equation}

In what follows, we assume that the kernel $K$ is continuous on a compact domain $D$ (the so-called Mercer kernel). Such kernels satisfy all
  the indicated above conditions (see \cite[Theorem~4.49]{StChr08}), and we even have (\ref{ptwise})
with absolute and uniform convergence on $D \times D$.
 Taking into account the spectral theorem \cite[Chapt.~4.5]{StChr08}, we have that the system $(e^*_n (\bx))_{n=1}^{\infty}$ of eigenfunctions of the non-negative compact self-adjoint operator $W_{\varrho_D}$ is an orthonormal basis in $H(K)$. Note, that $( \eta_n (\bx) )_{n=1}^{\infty}$ is also ONB in $L_{2}(D, \varrho_D)$.

Hence, for every $f\in H(K)$  with a Mercer kernel $K$, it holds
\be\label{f_series_in_H(K)}
f(\bx) = \sum_{k=1}^{\infty} (f, e^*_k)_{H(K)} e^*_k (\bx),
\ee
where $(f, e^*_k)_{H(K)}$, $k=1, 2, \dots$, are the Fourier coefficients of $f$ with respect to the system $( e^*_n (\bx) )_{n=1}^{\infty}$. Let us further denote
$$
		P_mf := \sum\limits_{k=1}^{m} (f,e^*_k)_{H(K)}e^*_k(\cdot)
$$
the projection onto the space $\operatorname{span}\{e^*_1(\cdot),...,e^*_{m}(\cdot)\}$.
Due to the Parseval's identity, the norm of the space $H(K)$ takes the following form
\be\label{norm_H(K)_sigma}
\|f\|^2_{H(K)} =   \sum_{k=1}^{\infty} | (f, e^*_k)_{H(K)} |^2 .
\ee

Note, that for $m\in \N$ the function
$$N_{K,\varrho_D}(m,\bx) = \sum_{k=1}^{m-1} \left| \eta_k (\bx) \right|^2$$
is often called ``Christoffel function'' in literature (see \cite{Gr19} and references therein).
For $V_m = \operatorname{span}\{\eta_1(\cdot),...,\eta_{m-1}(\cdot)\}$, we define
\be\label{N(m)}
N_{K,\varrho_D}(m) := \sup_{\bx \in D}N_{K,\varrho_D}(m,\bx) =
 \sup\limits_{\substack{f \in V_m\\ f\neq 0}} \|f\|_{\infty}^2/\|f\|^2_2\,.
 \ee
We may use the quantity $N_{K,\varrho_D}(m)$ to bound the operator norm $\|\Id-P_{m-1}\|_{K,\infty}$ (We use the abbreviation $\|A\|_{K,\infty} := \|A\|_{H(K) \to \ell_{\infty}(D)}$). It holds
$$
	\sup\limits_{\|f\|_{H(K)} \leq 1}\|f-P_{m-1}f\|_{\infty} = \sup\limits_{\bx \in D}\sup\limits_{\|f\|_{H(K)}\leq 1}
	|f(x)-P_{m-1}(x)| = \sup\limits_{\bx \in D}\sup\limits_{\|f\|_{H(K)}\leq 1}
	\Big|\sum\limits_{k\geq m} (f,e_k^*(\cdot))_{H(K)}e_k^{*}(\bx)\Big|\,.
$$
This implies
$$
	\|\Id-P_{m-1}\|^2_{K,\infty} = \sup\limits_{\bx \in D}\sum\limits_{k\geq m}|e_k^*(\bx)|^2\,.
$$
Using, that $e_k^*(\bx) = \sigma_k \eta_k^*$, where the sequence $( \sigma_n)_{n=1}^{\infty}$ is non-increasing, i.e.,
for all $2^l \leq k < 2^{l+1}$ it holds $\sigma_k^2 \leq  \frac{1}{2^{l-1}} \sum\limits_{ 2^{l-1}\leq j <2^l}\sigma_j^2$,
we obtain by a straight-forward calculation
$$
\|\Id-P_{m-1}\|^2_{K,\infty} = \sup\limits_{\bx \in D} \sum_{l= \lfloor \log_2 m \rfloor}^{\infty} \sum_{2^l \leq k < 2^{l+1}} \left| \sigma_k \eta_k^* (\bx) \right|^2
\leq \sum_{l= \lfloor \log_2 m \rfloor}^{\infty}  \frac{ N_{K,\varrho_D}(2^{l+1}) }{2^{l-1}} \sum_{ 2^{l-1}\leq j <2^l}\sigma_j^2 .
$$
Hence, in view of the relations $2^{l+1} \leq 4j$ and $1/2^{l-1} < 2/j$, that are true for all $2^{l-1}\leq j <2^l$, we get
\begin{align}\label{}
\|\Id-P_{m-1}\|^2_{K,\infty} & \leq
 \sum_{l= \lfloor \log_2 m \rfloor}^{\infty} \sum_{ 2^{l-1}\leq j <2^l} \frac{2 N_{K,\varrho_D}(4j)}{j} \sigma_j^2
 \leq 2   \sum_{k= \lfloor 2^{\log_2 m-1} \rfloor }^{\infty} \frac{N_{K,\varrho_D}(4k)}{k} \sigma_k^2
 \nonumber \\
 & = 2\sum\limits_{k\geq \lfloor m/2 \rfloor}\frac{N_{K,\varrho_D}(4k)\sigma_k^2}{k}.
\label{op_norm_infty}
\end{align}

Note, that the system $( \eta_k (\bx) )_{k=1}^{\infty}$ is not necessarily a uniformly $\ell_\infty$-bounded ONS in $L_2(D,\varrho_D)$, nevertheless
the equality (\ref{ptwise}) is true.
For systems $( \eta_k (\bx) )_{k=1}^{\infty}$, where for all $k\in \N$
$$
\| \eta_k (\bx) \|_{\ell_\infty(D)} \leq B, \quad k\in \N,
$$
we have
\be\label{N-T_BOS_eta}
N_{K,\varrho_D}(m) \leq (m-1) B^2.
\ee

\section{Least squares approximation in the uniform norm}

Let further $\boldsymbol{\rm X} = \{ \boldsymbol{\rm x}^1, \dots,  \boldsymbol{\rm x}^n \}$,
$\boldsymbol{\rm x}^i = ({\rm x}_1^i, \dots, {\rm x}_d^i) \subset D$, $i=1, \dots, n$, be the drawn independently nodes according to a Borel measure $\varrho_D$ on $D$ that has full support.
In \cite{KUV19} a recovery operator $S^{m}_{\boldsymbol{\rm X}}$ was constructed which computes the best least squares fit $S^{m}_{\boldsymbol{\rm X}}f$ to the given data $\mathbf{f}(\boldsymbol{\rm X}) :=(f(\boldsymbol{\rm x}^1), \dots, f(\boldsymbol{\rm x}^n))$ from the finite-dimensional space spanned by $\eta_k (\bx)$, $k=1, \dots, m-1$.

Namely, let $\boldsymbol{\rm f}:= (f(\boldsymbol{\rm x}^1), \dots, f(\boldsymbol{\rm x}^n))^{\top}$, $\boldsymbol{\rm c} := (c_1, \dots, c_{m-1})^{\top}$,
$( \eta_k (\bx) )_{k=1}^{\infty}= ( \sigma_k^{-1} e^*_k(\bx) ))_{k=1}^{\infty}$ and
\be\label{matrix_Lm}
{\boldsymbol{\rm L}}_{n,m} := \bL_{n,m} ( \boldsymbol{\rm X} ) =
\begin{pmatrix}
\eta_1 (\boldsymbol{\rm x}^1) &  \eta_2 (\boldsymbol{\rm x}^1) & \cdots &  \eta_{m-1} (\boldsymbol{\rm x}^1) \\
\vdots  & \vdots  & \ddots & \vdots  \\
\eta_1 (\boldsymbol{\rm x}^n) &  \eta_2 (\boldsymbol{\rm x}^n) & \cdots &  \eta_{m-1} (\boldsymbol{\rm x}^n)
\end{pmatrix}.
\ee
We should solve the over-determined linear system ${\boldsymbol{\rm L}}_m \cdot \boldsymbol{\rm c} = \boldsymbol{\rm f}$
via least squares (e.g.\ directly or via the LSQR algorithm \cite{PaSa82}),
 i.e., compute
  	$$
  	\boldsymbol{\rm c} := (\bL_m^{\ast}\bL_m)^{-1} \, \bL_m^{\ast}\cdot \mathbf{f}.
  	$$
The output $\boldsymbol{\rm c} \in \C^{m-1}$ gives the coefficients of the approximant
\be\label{algorithm_Sm}
  S^{m}_{\boldsymbol{\rm X}}f:= \sum_{k=1}^{m-1} c_k \eta_k.
\ee

Further, let us consider a weighted least squares algorithm (see below) using the following weight/density function, which was introduced by D. Krieg and M. Ullrich in \cite{KrUl19}. We set
\be\label{density_f}
\varrho_m (\bx) = \frac{1}{2}\Big(\frac{1}{m-1}\sum_{k=1}^{m-1} \left| \eta_k (\bx) \right|^2 + \frac{1}{\sum_{k=m}^{\infty} \lambda_k}\sum\limits_{k=m}^{\infty} |e_k^{\ast}(\bx)|^2\Big).
\ee
This density first appeared in \cite{KrUl19}. We will also use the below algorithm with the simpler density function
\begin{equation}\label{sd}
	\varrho'_m(\bx) = \frac{1}{2(m-1)}\sum\limits_{k=1}^{m-1}|\eta_k(\bx)|^2 + \frac{1}{2}\,,
\end{equation}
in case $\varrho_D$ is a probability measure (see Remark \ref{remark_other_density}). This density function has been implicitly used by V.N. Temlyakov in \cite{Teml2020} and by Cohen, Migliorati \cite{CoMi16}. In many relevant cases both densities yield similar bounds. However, as we point out in Remark \ref{rem_bern} it could make a big difference even in the univariate situation.

\begin{algorithm}[H]
\caption{Weighted least squares regression \cite{KrUl19}, \cite{CoMi16}.}\label{algo1:reweighted}
  \begin{tabular}{p{1.2cm}p{4.5cm}p{8.9cm}}
    Input: & $\bX = (\bx^1,...,\bx^n)\in D^n$ \hfill & set of distinct sampling nodes, \\
      & $\mathbf{f} = (f(\bx^1),...,f(\bx^n))^\top$ \hfill & samples of $f$ evaluated at the nodes from $\bX$, \\
      & $m\in\N$ & $m < n$ such that the matrix $\widetilde{\bL}_{k,m}$ in \eqref{eq:tilde_L} has full (column) rank.
  \end{tabular}
  \begin{algorithmic}
	\STATE
      Compute reweighted samples $\boldsymbol{g}:=(g_j)_{j=1}^n$ with $g_j:=\begin{cases}0,  & \varrho_m(\bx^j)=0,\\
      f(\bx^j)/\sqrt{\varrho_m(\bx^j)}, & \varrho_m(\bx^j)\neq 0\,.
       \end{cases}$

  	\STATE
  	Solve the over-determined linear system
  	\begin{equation}
  	\widetilde{\bL}_{k,m} \cdot (\tilde{c}_1,...,\tilde{c}_{m-1})^\top = \mathbf{g}\,, \; \widetilde{\bL}_{k,m}:=\Big(l_{j,k}\Big)_{j=1,k=1}^{n,m-1},\; l_{j,k}:=\begin{cases}0,  & \varrho_m(\bx^j)=0,\\
  	      \eta_k(\bx^j)/\sqrt{\varrho_m(\bx^j)}, & \varrho_m(\bx^j)\neq 0,
  	       \end{cases}
  	       \label{eq:tilde_L}
  	\end{equation}
  	via least squares (e.g.\ directly or via the LSQR algorithm \cite{PaSa82}), i.e., compute
  	$$
  	(\tilde{c}_1,...,\tilde{c}_{m-1})^\top := (\widetilde{\bL}_{k,m}^{\ast}\widetilde{\bL}_{k,m})^{-1} \,\widetilde{\bL}_{k,m}^{\ast}\cdot \mathbf{g}.
  	$$
  \end{algorithmic}
   Output:  $\mathbf{\tilde{c}} = (\tilde{c}_1,...,\tilde{c}_{m-1})^\top\in \C^{m-1}$ coefficients of the approximant $\widetilde{S}_{\bX}^m f:=\sum_{k = 1}^{m-1} \tilde{c}_k \eta_k$.
  \end{algorithm}

\begin{theorem}\label{prob_l_infty_general}
	Let $H(K)$ be a reproducing kernel Hilbert space of complex-valued functions defined on a compact domain $D \subset \R^d$ with a continuous and bounded kernel $K(\bx, \by)$, $K\colon D\times D \rightarrow \C$, i.e., $\|K\|_{\infty} := \sup_{\bx \in D} \sqrt{K(\bx,\bx)} < \infty$, and $\varrho_D$ be a finite Borel measure with full support on $D$.
We denote with $( \sigma_n )_{n=1}^{\infty}$ the non-increasing sequence of singular numbers of the embedding
 $\Id \colon H(K) \rightarrow L_{2}(D, \varrho_D)$. The points $\bX = (\bx^1,...,\bx^n)$ are drawn i.i.d. with respect to the measure defined by the density $\varrho_m(\cdot)d\varrho_D$. There are absolute constants $c_1,c_2,c_3>0$ such that for
 $$
	m:= \Big\lfloor \frac{n}{c_1r\log n}\Big\rfloor
 $$
with $r>1$, the reconstruction operator $\widetilde{S}_{\bX}^m$ (see Algorithm~\ref{algo1:reweighted}), satisfies
\be\label{prob_l_infty_first_gen}
\begin{split}
\Prob &\Big(\sup\limits_{\|f\|_{H(K)}\leq 1}\|f-  \widetilde{S}_{\bX}^m f\|^2_{\ell_{\infty}(D)}\leq c_3 \max\Big\{\frac{N_{K,\varrho_D}(m)}{m}\sum\limits_{k\geq \lfloor m/2 \rfloor}\sigma_k^2,
\sum\limits_{k\geq \lfloor m/2 \rfloor}\frac{N_{K,\varrho_D}(4k)\sigma_k^2}{k}\Big\}\\
&\geq 1- c_2n^{1-r}\,.
\end{split}
\ee

\end{theorem}
\begin{proof} Let $f \in H(K)$ such that $\|f\|_{H(K)}\leq 1$.  For the reconstruction operator $\widetilde{S}_{\bX}^m f$ from Algorithm~\ref{algo1:reweighted} it holds
 \be\label{infty_general}
\| f-  \widetilde{S}_{\bX}^m f \|_{\ell_{\infty}(D)}  \leq
\| f- P_{m-1}f \|_{\ell_{\infty}(D)} +
\|P_{m-1}f - \widetilde{S}_{\bX}^m f \|_{\ell_{\infty}(D)} ,
\ee
where the projection $P_{m}f$ is defined in Section~\ref{section_RKHS}.

From (\ref{op_norm_infty}) we obtain that
\be\label{sup_f-Pm}
\sup\limits_{\|f\|_{H(K)}\leq 1} \| f- P_{m-1}f \|_{\ell_{\infty}(D)} \leq
\sqrt{ 2 \sum\limits_{k\geq \lfloor m/2 \rfloor}\frac{N_{K,\varrho_D}(4k)\sigma_k^2}{k}}.
\ee

Clearly, $\widetilde{S}^{m}_{\boldsymbol{\rm X} } (P_{m-1}f) = P_{m-1}f$.
Therefore, for the second summand on the right-hand side of (\ref{infty_general})
we can write
\begin{align}
\|P_{m-1}f - \widetilde{S}_{\bX}^m f \|_{\ell_{\infty}(D)}  & =
\|\widetilde{S}_{\bX}^m  \left(f - P_{m-1}f \right) \|_{\ell_{\infty}(D)}
= \left\| \sum_{k=1}^{m-1}  \tilde{c}_k  \eta_k (\bx) \right\|_{\ell_{\infty}(D)}
\nonumber \\
& \leq  \sqrt{N_{K,\varrho_D}(m)} \left( \sum_{k=1}^{m-1} | \tilde{c}_k |^2 \right)^{1/2}.
\label{second_term_linfty}
\end{align}

Let us estimate the quantity $\left( \sum_{k=1}^{m-1} | \tilde{c}_k |^2 \right)^{1/2}$, i.e.,  $\ell_2$-norm of the coefficients of
the operator $\widetilde{S}_{\bX}^m$. We have
\begin{align}
\left( \sum_{k=1}^{m-1} | \tilde{c}_k |^2 \right)^{1/2} & =
\left\| (\widetilde{\bL}_{k,m}^{\ast}\widetilde{\bL}_{k,m})^{-1} \,\widetilde{\bL}_{k,m}^{\ast}
\left(
  \begin{array}{c}
    \frac{\left(f - P_{m-1}f \right) ( \bx^1 )}{ \sqrt{\varrho_m(\bx^1)} } \\
        \cdots \\
        \frac{\left(f - P_{m-1}f \right) ( \bx^n )}{ \sqrt{\varrho_m(\bx^n)} } \\
                \end{array}
 \right)
 \right\|_{\ell_2}
\nonumber \\
 & \leq \left\| (\widetilde{\bL}_{n,m}^{\ast}\widetilde{\bL}_{n,m})^{-1} \,\widetilde{\bL}_{n,m}^{\ast} \right\|_{2 \to 2}
\left( \sum_{i=1}^{n} \frac{ | f(\bx^i) - P_{m-1}f (\bx^i) |^2  }{ \varrho_m(\bx^i) } \right)^{1/2}.
 \label{estim_ck_l2}
\end{align}
Now we estimate the spectral norm. For the function
$$
\widetilde{N} (m) := \sup_{\bx \in D} \sum_{k=1}^{m-1} \frac{ \left| \eta_k (\bx) \right|^2}{ \varrho_m (\bx) }\,,
$$
by the choice (\ref{density_f}) of $\varrho_m$, it holds
$\widetilde{N} (m)  \leq 2(m-1)$.
Hence, in view of \cite[Theorem~5.1]{Ni_Sc_Ut20},  we get with
$$
\widetilde{N} (m) \leq 2(m-1) \leq \frac{n}{10r\log n }
$$
for $r>1$ the  condition
\be\label{condition_m}
m \leq \frac{n}{c_1 r\log n }\,,
\ee
and have
\be\label{spectral_norm}
\left\| (\widetilde{\bL}_{n,m}^{\ast}\widetilde{\bL}_{n,m})^{-1} \,\widetilde{\bL}_{n,m}^{\ast} \right\|_{2 \to 2} \leq \sqrt{ \frac{2}{n} }
\ee
with probability at least $1-3n^{1-r}$.
Combining  (\ref{estim_ck_l2}) and (\ref{spectral_norm}),  and taking into account (\ref{f_series_in_H(K)}), we derive to
\begin{align}
\sum_{k=1}^{m-1} | \tilde{c}_k |^2  & \leq \frac{2}{n}  \sum_{i=1}^{n} \frac{ | f(\bx^i) - P_{m-1}f (\bx^i) |^2  }{ \varrho_m(\bx^i) }
\nonumber \\
& = \frac{2}{n}  \sum_{i=1}^{n} \left| \sum_{k=m}^{\infty} (f, e^*_k)_{H(K)} \frac{e^*_k (\bx^i)}{ \sqrt{\varrho_m(\bx^i)} } \right|^2
= \frac{2}{n} (\widetilde{\mathbf{\Phi}} \mathbf{\hat{f}}, \widetilde{\mathbf{\Phi}} \mathbf{\hat{f}})_{\ell_2},
\label{sum_c_tilde_1}
\end{align}
where
$$
\widetilde{\mathbf{\Phi}} := \left( \frac{e^*_k (\bx^i)}{ \sqrt{\varrho_m(\bx^i)} } \right)_{i=1, k=m}^{n, \infty} = \left(
                                                                                                                       \begin{array}{c}
                                                                                                                         \by^1 \\
                                                                                                                         \vdots \\
                                                                                                                         \by^n \\
                                                                                                                       \end{array}
                                                                                                                     \right)
$$
is an infinite matrix with $\by^i := 1/\sqrt{\varrho_m(\bx^i)} (e^*_m (\bx^i), e^*_{m+1} (\bx^i), \dots )$, $i=1,\dots, n$,
$\mathbf{\hat{f}} := ( (f, e^*_k)_{H(K)}  )_{k=m}^{\infty}$.
Then, (\ref{sum_c_tilde_1}) yields
\begin{align}
\sum_{k=1}^{m-1} | \tilde{c}_k |^2  & \leq
\frac{2}{n} (\widetilde{\mathbf{\Phi}}^* \widetilde{\mathbf{\Phi}} \mathbf{\hat{f}},  \mathbf{\hat{f}})_{\ell_2}
\leq \frac{2}{n} \left\|  \widetilde{\mathbf{\Phi}}^* \widetilde{\mathbf{\Phi}} \right\|_{2 \to 2}  ( \mathbf{\hat{f}},  \mathbf{\hat{f}})_{\ell_2}
\leq \frac{2}{n} \left\|  \widetilde{\mathbf{\Phi}}^* \widetilde{\mathbf{\Phi}} \right\|_{2 \to 2} \| f \|^2_{H(K)}
\nonumber \\
& \leq 2 \left( \left\|  \frac{1}{n}\widetilde{\mathbf{\Phi}}^* \widetilde{\mathbf{\Phi}} - \mathbf{\Lambda} \right\|_{2 \to 2} +
\|\mathbf{\Lambda}  \|_{2 \to 2} \right)
\label{sum_c_tilde_2}
\end{align}
with $\mathbf{\Lambda}  := \diag (\sigma_m^2, \sigma_{m+1}^2, \dots)$.

We can use now \cite[Prop.\ 3.8]{MoUll20}, in view of the fact that
$$
\| \by^i \|^2_2 \leq  \sup_{\bx \in D} \sum_{k=m}^{\infty} \frac{\left|e^*_k(\bx)\right|^2}{\varrho_m(\bx)}
\leq 2 \sum\limits_{k=m}^{\infty}\lambda_k  =: M^2.
$$

We get with high probability that
$$
\left\|  \frac{1}{n}\widetilde{\mathbf{\Phi}}^* \widetilde{\mathbf{\Phi}} - \mathbf{\Lambda} \right\|_{2 \to 2}
\leq \max \left\{ \frac{8 r \log n}{n} M^2 \kappa^2, \| \mathbf{\Lambda} \|_{2 \to 2} \right\}:= F,
$$
where $\kappa=\frac{1+\sqrt{5}}{2}$, $\| \mathbf{\Lambda} \|_{2 \to 2} = \sigma_m^2$.
Therefore, from (\ref{sum_c_tilde_2}) we get
\begin{align}
 \sum_{k=1}^{m-1} | \tilde{c}_k |^2  & \leq 2(F + \sigma_m^2)
= 2 \left( \max \left\{ \frac{16 \kappa^2 r \log n}{n}  \sum\limits_{k=m}^{\infty}\lambda_k , \sigma_m^2 \right\} + \sigma_m^2 \right)
\nonumber \\
& \leq  4 \max \left\{ \frac{16 \kappa^2 r \log n}{n}  \sum\limits_{k=m}^{\infty}\lambda_k , \sigma_m^2 \right\}.
 \label{sum_c_tilde_final}
\end{align}
Choosing the maximal $m$ that satisfy the relation (\ref{condition_m}), i.e., $m= \big\lfloor \frac{n}{c_1 r\log n}\big\rfloor \,,$
we get
$$
\frac{1}{m} = \frac{1}{\big\lfloor \frac{n}{c_1 r\log n}\big\rfloor } \geq
\frac{1}{ \frac{n}{c_1 r\log n} } = \frac{c_1 r\log n}{n}\, ,
$$
that is
$$
 \frac{r \log n}{n} \leq \frac{1}{c_1 m}\,,
$$
and hence,
\begin{align}
\|P_{m-1}f - \widetilde{S}_{\bX}^m f \|_{\ell_{\infty}(D)} & \leq 2  \sqrt{N_{K,\varrho_D}(m)}
 \max \left\{ \frac{4 \kappa}{\sqrt{c_1}} \sqrt{\frac{\sum_{k=m}^{\infty}\sigma_k^2}{m}}, \sigma_m \right\}
 \nonumber
 \\
 & \leq C'  \sqrt{ \frac{N_{K,\varrho_D}(m)}{m} \sum\limits_{k\geq \lfloor m/2 \rfloor}\sigma_k^2} \,.
 \label{sup_Pm-Smf}
\end{align}

Finally, from (\ref{infty_general}), (\ref{sup_f-Pm}), and (\ref{sup_Pm-Smf}) we obtain
\be\label{Th1_constants_last_line}
\sup\limits_{\|f\|_{H(K)}\leq 1}\|f-  \widetilde{S}_{\bX}^m f\|_{\ell_{\infty}(D)} \leq
\sqrt{ 2 \sum\limits_{k\geq\lfloor m/2 \rfloor} \frac{N_{K,\varrho_D}(4k)\sigma_k^2}{k}} +
C'  \sqrt{ \frac{N_{K,\varrho_D}(m)}{m} \sum\limits_{k\geq \lfloor m/2 \rfloor}\sigma_k^2}\,,
\ee
which yields the statement of Theorem~\ref{prob_l_infty_general}.
\end{proof}

\begin{remark}\label{remark_constants}
Note, that we can get explicit values for the constants $c_1, c_2, c_3$ in Theorem~\ref{prob_l_infty_general}. Indeed, following the proof we see, that  $c_2 =3$.
As to $c_1, c_3$, the relation (\ref{sup_Pm-Smf}) yields
$$
 \|P_{m-1}f - \widetilde{S}_{\bX}^m f \|_{\ell_{\infty}(D)} \leq
2\left( \frac{4 \kappa}{\sqrt{c_1}} +1 \right) \sqrt{N_{K,\varrho_D}(m)} \max \left\{ \sqrt{\frac{\sum_{k=m}^{\infty}\sigma_k^2}{m}}, \sigma_m \right\}.
$$
Then, taking into account (see Remark~2 on p. 16 of \cite{Ni_Sc_Ut20}) that
$$
\sigma_m^2 +  \frac{1}{m} \sum_{k=m}^{\infty}\sigma_k^2 \leq \frac{2}{m} \sum_{k\geq\lfloor m/2 \rfloor}\sigma_k^2,
$$
we get
$$
 \|P_{m-1}f - \widetilde{S}_{\bX}^m f \|_{\ell_{\infty}(D)} \leq
2\sqrt{2} \left( \frac{4 \kappa}{\sqrt{c_1}} +1 \right) \sqrt{\frac{N_{K,\varrho_D}(m)}{m} \sum_{k\geq\lfloor m/2 \rfloor}\sigma_k^2 },
$$
where $\kappa=\frac{1+\sqrt{5}}{2}$. Respectively, we put $C' = 2\sqrt{2} \left( \frac{4 \kappa}{\sqrt{c_1}} +1 \right)$ in (\ref{sup_Pm-Smf}), and from (\ref{Th1_constants_last_line}) get
\begin{align*}
\sup\limits_{\|f\|_{H(K)}\leq 1}\|f & -  \widetilde{S}_{\bX}^m f\|_{\ell_{\infty}(D)} \leq
(\sqrt{2} + C')
 \left( \sqrt{\sum\limits_{k\geq\lfloor m/2 \rfloor} \frac{N_{K,\varrho_D}(4k)\sigma_k^2}{k}} +
  \sqrt{ \frac{N_{K,\varrho_D}(m)}{m} \sum\limits_{k\geq \lfloor m/2 \rfloor}\sigma_k^2} \right)
\\
&\leq  2 \sqrt{2}\left( \frac{8 \kappa}{\sqrt{c_1}} +3 \right)
  \max\Big\{ \sqrt{\sum\limits_{k\geq\lfloor m/2 \rfloor} \frac{N_{K,\varrho_D}(4k)\sigma_k^2}{k}},
\sqrt{ \frac{N_{K,\varrho_D}(m)}{m} \sum\limits_{k\geq \lfloor m/2 \rfloor}\sigma_k^2}\Big\}.
\end{align*}
Hence,
$c_3 = 8 \left( \frac{8 \kappa}{\sqrt{c_1}} + 3 \right)^2$ in the general estimate (\ref{prob_l_infty_first_gen}).

For arbitrary orthonormal systems $( \eta_k (\bx) )_{k=1}^{\infty}$,
we can take $c_1 = 20$ in (\ref{condition_m}) and obtain then $c_3 = 278$.
If $( \eta_k (\bx) )_{k=1}^{\infty}$ is uniformly $\ell_\infty$-bounded ONS in $L_2(D,\varrho_D)$,
then we can use the plain (non-weighted) least squares algorithm and take $c_1 = 10$, hence,
put $c_3 = 403$ (see Section~\ref{section_examples} for further comments).
 \end{remark}

\begin{corollary}\label{cor} If there is a measure $\varrho_D$ on $D$ such that $N_{K,\varrho_D}(k) = \mathcal{O}(k)$ we obtain \be
\begin{split}
 \sup\limits_{\|f\|_{H(K)}\leq 1}\| f-  \widetilde{S}_{\bX}^mf \|^2_{\ell_{\infty}(D)}\leq C_{\varrho_D,K} \sum\limits_{k\geq \lfloor m/2 \rfloor} \sigma_k^2 \leq C_{\varrho_D,K}
 a_{\lfloor m/2 \rfloor}(\Id_{K,\infty})^2
 \end{split}
\ee
for constant $C_{K,\varrho_D}>0$ depending on the kernel $K$ and the measure $\varrho_D$.
\end{corollary}

\begin{proof} The statement follows directly from Theorem \ref{prob_l_infty_general} in combination \cite[Lem.\ 3.3]{CoKuSi16}.

\end{proof}

\begin{remark}[Other density]\label{remark_other_density}
Instead of \eqref{density_f} one could also use the simpler density function \eqref{sd} defined above
in case that $\varrho_D$ is a probability measure on $D$ (otherwise with a proper normalization). Then the estimate would be slightly worse. Instead of the term
$$
\frac{N_{K,\varrho_D}(m)}{m}\sum\limits_{k\geq \lfloor m/2 \rfloor}\sigma_k^2
$$
in (\ref{prob_l_infty_first_gen}) we obtain
$$
\frac{N_{K,\varrho_D}(m)}{m}\sum\limits_{k\geq  \lfloor cm \rfloor}\frac{N_{K,\varrho_D}(4k)\sigma_k^2}{k}\, ,
$$
where $c>0$ is an absolute constant. This does not make a difference if $N_{K,\varrho_D}(m) = \mathcal{O}(m)$. However, it makes a difference if we have for instance $N_{K,\varrho_D}(m)$ growing as $m^2$ which is the case for, e.g., Legendre polynomials. See Example~\ref{example_Bernardi} below.
\end{remark}

\begin{remark}[Non-weighted version]\label{remark_Non-weighted}
Note, that one could even use a non-weighted least squares algorithm
(see (\ref{algorithm_Sm})), without considering the additional density function $\varrho_m$.
If the measure $\varrho_D$ is fixed, namely the one where the data is coming from, then we
get a bound
$$
  \sup\limits_{\|f\|_{H(K)}\leq 1} \| f - S_{\bX}^m f\|^2_{\ell_{\infty}(D)} \leq
  C \sum\limits_{k \geq  \lfloor m^*/2 \rfloor } \frac{N_{K,\varrho_D}(4k) \sigma_k^2}{k} \, ,
$$
where $C>0$ is an absolute constant which can be calculated precisely,
 $m^*$ is the largest number such that $N_{K,\varrho_D}(m^*) \leq n/(10 r\log n)$.
\end{remark}

\section{Improved bounds for sampling numbers}\label{section_improved}

By exactly the same strategy as used in the recent paper by
N.~Nagel, M.~Sch\"{a}fer and T.~Ullrich~\cite{Ni_Sc_Ut20}, we can improve the bound (\ref{prob_l_infty_first_gen})
 for sampling numbers associated to the compact embedding of RKHS with Mercer kernel
into the space of bounded on $D$ functions,
 applying
a modification of the Weaver sub-sampling strategy.

\begin{theorem}[\cite{Ni_Sc_Ut20, NiOlUl16,LimTe2020}]\label{th_NSU}
     Let $k_1, k_2, k_3 > 0$ and $\mathbf{u}_1, ..., \mathbf{u}_n \in
\C^m$ with $\|\mathbf{u}_i\|_2^2 \leq k_1 \frac{m}{n}$ for all $i=1,
..., n$ and
     $$
     k_2 \|\mathbf{w}\|_2^2 \leq \sum_{i=1}^n |\langle \mathbf{w},
\mathbf{u}_i\rangle|^2 \leq k_3 \|\mathbf{w}\|_2^2
     $$
     for all $\mathbf{w} \in \C^m$. Then there is a $J \subseteq [n]$ of
size $\# J \leq C_1 m$ with
     $$
     C_2 \cdot \frac{m}{n} \|\mathbf{w}\|_2^2 \leq \sum_{i \in J}
|\langle \mathbf{w}, \mathbf{u}_i\rangle|^2 \leq C_3 \cdot \frac{m}{n}
\|\mathbf{w}\|_2^2
     $$
     for all $\mathbf{w} \in \C^m$, where $C_1, C_2, C_3$ only depend on
$k_1, k_2, k_3$. More precisely, we can choose
     \[
     C_1 = 1642 \frac{k_1}{k_2} \,,\quad
     C_2 = (2+\sqrt{2})^2k_1 \,,\quad
     C_3 = 1642 \frac{k_1k_3}{k_2}
     \]
     in case $\frac{n}{m}\geq 47 \frac{k_1}{k_2}$. In the regime
$1\le\frac{n}{m}< 47 \frac{k_1}{k_2}$ one may put $C_1 = 47k_1/k_2$,
$C_2 = k_2$, $C_3=47k_1 k_3/k_2$.
\end{theorem}

\begin{theorem}\label{main_samp} Under the assumptions of Theorem \ref{prob_l_infty_general} we obtain the following bounds for the sampling numbers $g_n(\Id:H(K) \to \ell_{\infty}(D))$. The measure $\varrho_D$ is at our disposal.

{\em (i)} There is an absolute constant $b>0$ such that
$$
g_{\lfloor bm\log m \rfloor}(\Id)^2  \leq c_3 \max\Big\{\frac{N_{K,\varrho_D}(m)}{m}\sum\limits_{k\geq \lfloor m/2 \rfloor}\sigma_k^2\,,
\sum\limits_{k\geq \lfloor m/2 \rfloor}\frac{N_{K,\varrho_D}(4k)\sigma_k^2}{k}\Big\}\,.
$$

{\em (ii)} There are absolute constants $c_4,c_5>0$ such that $$
 g_{m}( \Id)^2 \leq
 c_4\max\Big\{\frac{N_{K,\varrho_D}(m)\log m}{m}\sum\limits_{k\geq  \lfloor c_5m  \rfloor}\sigma_k^2\,,
\sum\limits_{k\geq  \lfloor c_5m  \rfloor}\frac{N_{K,\varrho_D}(4k)\sigma_k^2}{k}\Big\}\,.
 $$
\end{theorem}
\begin{proof} The statement (i) follows immediately from Theorem~\ref{prob_l_infty_general}.

To get (ii), we apply a sampling operator $\widetilde{S}_{J}^m $ using $O(m)$ sampling nodes
$(\bx^i)_{i\in J} \subset \bX$ constructed
 out of a random draw in combination with a sub-sampling procedure,
instead of $O(m \log m)$ (as in Algorithm~\ref{algo1:reweighted}).  That is, we omit
some rows in the matrix $\bL_{n,m}$ from (\ref{matrix_Lm})
(respectively, in $\widetilde{\bL}_{n,m}$)
by constructing  a sum-matrix $\widetilde{\bL}_{J,m}$ having $\# J = O(m)$ rows, such that
for all $\mathbf{w} \in \C^{m-1}$ (see Theorem~\ref{th_NSU}) it holds
$$
c_2 \|\mathbf{w}\|_2^2 \leq \frac{1}{m} \| \widetilde{\bL}_{J,m} \mathbf{w} \|_2^2
\leq C_2 \|\mathbf{w}\|_2^2 \,.
$$
With this matrix we perform the least squares method, applied to the shrinked vector of function
samples $(f(\bx^i))_{i\in J}$, see the proof of Theorem~5 in~\cite{Ni_Sc_Ut20} for the details.
\end{proof}

\begin{corollary}\label{cor2} If there is a measure $\varrho_D$ on $D$ such that $N_{K,\varrho_D}(k) = \mathcal{O}(k)$ we obtain
\be\label{cor2_estimate}
	g_m(\Id) \leq C_{\varrho_D,K}\min\{a_{\lfloor m/(c_6 \log m)\rfloor}(\Id)\,,\sqrt{\log m}\cdot a_{\lfloor c_5 m \rfloor}(\Id)\}\,.
\ee
\end{corollary}
\begin{proof}
In the case $N_{K,\varrho_D}(k) = \mathcal{O}(k)$, from Theorem~\ref{main_samp} we obtain
$$
g_m(\Id)^2 \leq C \min \Big\{\sum\limits_{k\geq \lfloor m/(c_6 \log m)\rfloor} \sigma_k^2\,,
 \log m  \sum\limits_{k\geq \lfloor c_5 m \rfloor}\sigma_k^2
\Big\}
$$
with an absolute constant $C>0$.
Taking into account, that
\begin{equation}\label{cobos}
 \sum_{j=n}^{\infty} \sigma_j^2  \leq
\varrho_D(D) \cdot a_n(\Id)^2
\end{equation}
(see  F.~Cobos, T.~K\"{u}hn and W.~Sickel~\cite[Lem.\ 3.3]{CoKuSi16}, which was proved in a more general case),
 we get (\ref{cor2_estimate}).
\end{proof}

\begin{remark}
The estimate from Corollary~\ref{cor2} improves the upper bound for sampling numbers of the embeddings of RKHS, consisting of
continuous periodic functions, into $L_{\infty} ([0,1)^d)$, that was established in a recent result by L.~K\"{a}mmerer, see \cite[Theorem~4.11]{Kam19}.
Note also, that this result was stated under stronger conditions on a weight function.
 In turn, it is an an improvement of the previous result by L.~K\"{a}mmerer and T.~Volkmer~\cite{Kam_Vol19}.
\end{remark}

\begin{remark} The estimate from (ii) of Theorem~\ref{main_samp} is non-constructive, since we have just
the fact of existence of such algorithm.
\end{remark}

\begin{remark} The bound in \eqref{cobos} appears at several places in the literature. See for instance R.~Kunsch~\cite[Prop.~2.7]{Kunsch2018}, K.~Yu.~Osipenko and O.~G.~Parfenov~\cite[Theorem~3]{Osipenko_Parfenov1995}, F.~Cobos, T.~K\"{u}hn and W.~Sickel~\cite[Lem.~3.3]{CoKuSi16} and also F.~Kuo, G.~W.~Wasilkowski and H.~Wo\'{z}niakowski~\cite{KWW08}.
\end{remark}

\section{Sampling and Kolmogorov numbers}

For a compact set $F \subset  \ell_{\infty}(D)$ of complex-valued functions defined on a compact domain $D \subset \R^d$, let us denote by
$d_m(\Id\colon F \rightarrow \ell_{\infty}(D))$ its $m$th Kolmogorov number, i.e.,
$$
 d_m(\Id) :=\inf\limits_{ V_m} \sup\limits_{ \| f \|_{F} \leq 1} \inf_{g \in V_m} \|  f - g  \|_{ \ell_{\infty}(D) },
$$
where $V_m$ is $m-1$-dimensional subspace of $F$. Assume also, that we know the optimal subspace for $d_m(\Id)$, i.e., a subspace  $V_m^*$ such that
$$
\sup\limits_{ \| f \|_{F} \leq 1} \inf_{g \in V_m^*} \|  f - g  \|_{ \ell_{\infty}(D) } \leq \frac{3}{2} d_m(\Id).
$$
Let further, $\varrho$ be a finite measure on $D$ and $(\phi_n)_{n=1}^{\infty}$ be any ONS in $V_m^*$ with respect to~$\varrho$. The following statement holds.
\begin{proposition}\label{samp_vs_Kolm}
Let $F$ be compactly embedded subspace (of complex-valued continuous functions) of $\ell_{\infty}(D)$, where  $D \subset \R^d$ is a compact subset. Then there is an absolute constant $b>0$ such that
 $$
g_{\lfloor bm\log m \rfloor}(\Id:F\to \ell_{\infty}(D)) \leq (2 +  \sqrt{ \varrho(D)/(m-1) }) \sqrt{N_{\varrho, V_m^*}} d_m(\Id)
$$
holds, where
\be\label{N_kolmog}
N_{\varrho, V_m^*} := \sup\limits_{f\in V_m^{*}}\frac{\|f\|^2_{\ell_{\infty}(D)}}{\|f\|^2_{L_2(D,\varrho)}},
\ee
 and the finite measure $\varrho$ on $D$ is at our disposal.
\end{proposition}
\begin{proof}
Following the proof of Theorem~\ref{prob_l_infty_general}, for $f \in F$, $g \in V_m^*$ and the recovery operator $\widetilde{S}_{\bX}^m f$
from Algorithm~\ref{algo1:reweighted} that uses the density function
$$
\varrho_m (\bx) = \frac{1}{2}\Big(\frac{1}{m-1}\sum_{k=1}^{m-1} \left| \phi_k (\bx) \right|^2 +  1 \Big),
$$
(compare to \eqref{sd}) we get with high probability, that
\begin{align}
\| f-  \widetilde{S}_{\bX}^m f \|_{\ell_{\infty}(D)} & \leq
\| f- g \|_{\ell_{\infty}(D)} + \|g - \widetilde{S}_{\bX}^m f \|_{\ell_{\infty}(D)}
= \| f- g \|_{\ell_{\infty}(D)} + \left\| \sum_{k=1}^{m-1}  \tilde{c}_k  \phi_k (\bx) \right\|_{\ell_{\infty}(D)}
\nonumber \\
& \leq \| f- g \|_{\ell_{\infty}(D)} + \sqrt{N_{\varrho, V_m^*}} \sqrt{ \frac{2}{n} \sum_{i=1}^{n} \frac{ | f(\bx^i) - g(\bx^i) |^2}{1/2} }
\nonumber \\
& \leq \| f- g \|_{\ell_{\infty}(D)} + 2\sqrt{N_{\varrho, V_m^*}} \| f-g \|_{ \ell_{\infty}(D) }.
\label{est_sampl_via_kolm_1}
\end{align}
Since for all $m \geq 2$, it holds
$$
m-1=\int_D  \sum\limits_{k=1}^{m-1}| \phi_k (\bx) |^2 d \varrho(\bx) \leq \varrho(D)\sup\limits_{\bx \in D}\sum\limits_{k=1}^{m-1}| \phi_k (\bx) |^2
$$
we have
$$
N_{\varrho, V_m^*} = \sup_{\bx \in D} \sum_{k=1}^{m-1} \left| \phi_k (\bx) \right|^2 \geq (m-1)/\varrho(D)\,.
$$
It follows from (\ref{est_sampl_via_kolm_1}) that
$$
\| f-  \widetilde{S}_{\bX}^m f \|_{\ell_{\infty}(D)} \leq
(2+\sqrt{ \varrho(D)/(m-1) }) \sqrt{N_{\varrho, V_m^*}} \| f-g \|_{ \ell_{\infty}(D) } .
$$
\end{proof}

\begin{remark}
Note, that Proposition~\ref{samp_vs_Kolm} can be extended
to more general function classes $F$, similar as in \cite{Teml2020} and \cite{KrUl20}.
\end{remark}

The following corollary is in the spirit of \cite[Theorem~1.1]{Teml2020}.
\begin{corollary}
There are universal constants $b, C >0$ such that
it holds
$$
g_{\lfloor bm \rfloor}(\Id) \leq C\sqrt{N_{\varrho, V_m^*}} d_m(\Id)\,, \quad m \geq \varrho(D)\,.
$$
\end{corollary}
\begin{proof}
This follows by an application of the subsampling technique in Theorem 4.1.
\end{proof}

\begin{remark}
It was shown by V.N.~Temlyakov~\cite[Theorem~1.1]{Teml2020}, that for a compact set $F$ of continuous on $D \subset \R^d$ complex-valued functions, there exist two positive constants $b$ and $B$ such that
$$
g_{\lfloor bm \rfloor}(\Id\colon F\rightarrow L_2(D, \varrho) ) \leq B d_m(\Id \colon F\rightarrow L_\infty(D, \varrho)).
$$
For special sets $F$ (in the reproducing kernel Hilbert space setting) the estimate
$$
g_{m}(\Id\colon F\rightarrow L_2(D, \varrho) )^2 \leq C \frac{\log m}{m} \sum\limits_{k \geq   \lfloor cm \rfloor } d_k(\Id \colon F\rightarrow L_2(D, \varrho))^2,
$$
with universal constants $C,c>0$,
was earlier proved by N.~Nagel, M.~Sch\"{a}fer and T.~Ullrich~\cite{Ni_Sc_Ut20} based on D.~Krieg and M.~Ullrich~\cite{KrUl19}. Similar results for non-Hilbert function spaces were obtained in the second part of research by D.~Krieg and M.~Ullrich~\cite{KrUl20}.
\end{remark}

\begin{remark}
The above bound on sampling numbers in terms of Kolmogorov numbers motivates the investigation of Kolmogorov numbers in certain situations. It further motivates the study of the Christoffel function of the related optimal subspaces. Currently we do not know how this bound can be made more explicit and therefore do not have convincing examples. We leave it as an open problem to control the quantity $\sqrt{N_{\varrho, V_m^*}}$ for the ``optimal'' subspace $V^*_m$. Note that the measure $\varrho$ is at our disposal.
\end{remark}

\begin{remark}[Approximation in the uniform norm]
In the recent papers by  V.N.~Temlyakov and T.~Ullrich~\cite{Teml_Ullrich2020_small, Teml_Ullrich2020_small_widths}, a behavior of some asymptotic characteristics of multivariate functions classes
in the uniform norm was studied in the case of ``small smoothness'' of functions. The crucial point here is the fact, that in a small smoothness setting the corresponding
approximation numbers are not square summable. The established estimates for Kolmogorov widths of the Sobolev classes of functions and classes  of functions with bounded mixed differences serve as a powerful tool to investigate sampling recovery problem in $L_2$ and $L_\infty$
(see also Section~7 in \cite{Teml_Ullrich2020_small} and Section~5 in \cite{Teml_Ullrich2020_small_widths}).
\end{remark}

\section{The power of standard information in the uniform norm}\label{section_power_st_inf}
Now we discuss the optimal order of convergence of algorithms that realize upper bounds for
approximation numbers $a_n(\Id)$ and sampling numbers $g_n(\Id)$ defined by (\ref{def_approxim_numbers}) and (\ref{def_sampling_numbers}), for the embedding operator
$\Id\colon H(K) \rightarrow F$.

 The optimal order of convergence $q^{\rm lin}_F$ among all algorithms that use linear information for the identity operator $\Id$
 is defined (see the monograph by E.~Novak and H.~Wo\'{z}niakowski~\cite[Chapt. 26.6.1]{NoWoIII}) as
$$
q^{\rm lin}_F := \sup \left\{ q \geq 0\colon \quad \lim_{n\rightarrow \infty} n^q a_n(\Id) =0 \right\},
$$
and, respectively, for standard information
$$
q^{\rm std}_F := \sup \left\{ q \geq 0\colon \quad \lim_{n\rightarrow \infty} n^q g_n(\Id) =0 \right\}.
$$
In the case $F=L_{2}(D, \varrho_D)$ we write
$ q^{ \rm lin }_{2, \varrho_D} := q^{\rm lin}_{L_{2}(D, \varrho_D)}$ and
$ q^{ \rm std }_{2, \varrho_D} := q^{\rm std}_{L_{2}(D, \varrho_D)}$, if $F=\ell_{\infty}(D)$ put
$ q^{ \rm lin }_{\infty} := q^{\rm lin}_{\ell_{\infty}(D)}$ and $ q^{ \rm std }_{\infty} := q^{\rm std}_{\ell_{\infty}(D)}$.
Note that, when investigating the approximation numbers $a_n(\Id)$, one can take $F=L_{\infty}(D)$ with $\|f\|_{L_{\infty}(D)}=
\operatorname{ess \, sup}_{\bx \in D} |f(\bx)|$ instead of considering the supremum norm of $\ell_{\infty}(D)$ space.

Let us introduce the following two assumptions:
\begin{enumerate}
\item[(i)] $N_{K,\varrho_D}(k) = \mathcal{O}(k^u)$;
  \item[(ii)] there exist $p> 1/2$, $C_2>0$ such that $\sigma_j \leq C_2 j^{-p}$, $j= 1,2,\dots$.
\end{enumerate}
Note that the values of $q$ and $p$ are optimal in the sense that
\begin{align*}
u &:= \inf\left\{u \colon \ N_{K,\varrho_D}(k) = \mathcal{O}(k^u) \right\},
\\
p &:= \sup\{p \colon \ \sigma_j \leq C_2 j^{-p}, \, j= 1,2,\dots\}.
\end{align*}

The stronger conditions on the integral operator were earlier introduced by F.~Kuo, G.~W.~Wasilkowski and H.~Wo\'{z}niakowski in \cite{KWW08, KWW09}. Namely, instead of the condition (i)
the authors assume that
\begin{enumerate}
\item[(i')] there exists $C_1>0$ such that $\| \eta_j \|_{ \ell_{\infty}(D)} \leq C_1$, $j= 1,2,\dots$,
\end{enumerate}
i.e., consider the special case  $N_{K,\varrho_D}(k) = \mathcal{O}(k)$. Note, that the constants $C_1, C_2$ can depend on the dimension $d$. We consider concrete examples of RKHS which satisfy the indicated conditions.

If the second assumption is true, then the optimal rate of convergence among all algorithms that use linear information in the worst case setting and error measured in
$L_{2}(D, \varrho_D)$-norm is $q^{ \rm lin }_{2, \varrho_D}=p$. If (i') and (ii) are satisfied, then, when
 switching from $L_{2}(D, \varrho_D)$ to $L_{\infty}(D)$ norm, we loose $1/2$ in the optimal order~\cite{KWW08}:
$$
q^{ \rm lin }_{\infty}= q^{ \rm lin }_{2, \varrho_D} - 1/2 = p- 1/2\,,
$$
where $p>1/2$. Clearly, we only need the weaker assumption (i) for such a statement.


Now we move to $\ell_{\infty}(D)$-error among all algorithms that use standard information in the worst case setting. It was proved in \cite{KWW09} that under the assumptions (i') and (ii) the optimal order of convergence
\begin{equation}\label{impr}
q^{ \rm std }_{\infty}\in \left[ \frac{2p}{2p+1} \left( p- \frac{1}{2} \right),  p- \frac{1}{2} \right].
\end{equation}
The upper bound can not be improved even for linear information. From Theorem~\ref{main_samp} we get the following improvement over \eqref{impr}.
 \begin{corollary} Suppose that (i) and (ii) hold true with $2p > u$. Then $q^{ \rm std }_{\infty} \geq p-u/2$. In case $u=1$ (or if (i') holds) we have $q^{ \rm std }_{\infty} = q^{\rm lin}_{\infty} = p-1/2$.
 \end{corollary}

\begin{proof} Due to the inequality
$$
\sum\limits_{k=n+1}^{\infty} k^{-\alpha}  \leq
  \int_n^{\infty} t^{-\alpha} {\rm d} t  = \frac{n^{-\alpha+1}}{\alpha-1}.
$$
This holds for any $\alpha >1$. We obtain
\begin{gather}
\max\Big\{\frac{N_{K,\varrho_D}(m)}{m}\sum\limits_{k\geq \lfloor m/2 \rfloor}\sigma_k^2,
\sum\limits_{k\geq \lfloor m/2 \rfloor}\frac{N_{K,\varrho_D}(4k)\sigma_k^2}{k}\Big\}
\nonumber \\
 \lesssim \max\Big\{m^{u-1} \sum\limits_{k\geq \lfloor m/2 \rfloor}k^{-2p},
\sum\limits_{k\geq \lfloor m/2 \rfloor} k^{u-1} k^{-2p} \Big\}
\nonumber \\
 \leq
\max\Bigg\{m^{u-1} \left( \lfloor m/2 \rfloor^{-2p} +  \frac{1}{2p-1} \lfloor m/2 \rfloor^{-2p+1} \right),
 \lfloor m/2 \rfloor^{-(2p-u+1)} +  \frac{1}{2p-u} \lfloor m/2 \rfloor^{-(2p-u)}
 \Bigg\}
 \nonumber \\
\leq \max\Bigg\{\frac{2p}{2p-1} m^{u-1} \lfloor m/2 \rfloor^{-2p+1},
 \frac{2p-u+1}{2p-u} \lfloor m/2 \rfloor^{-(2p-u)} \Bigg\}
\lesssim_{p} m^{-(2p-u)}.
\label{power_1}
\end{gather}
Respectively, for $n = \lfloor bm\log m \rfloor$,  where $b>0$ is an absolute constant, the estimate (i) of Theorem~\ref{main_samp} yields
$$
g_{n}(\Id)  \lesssim_p n^{-(p-u/2)} (\log n)^{p-u/2}.
$$

We proved now, that under the considered assumptions $q^{ \rm std }_{\infty}\geq p-u/2$. In case $q = 1$, i.e., where (i) = (i') we have that  $q^{ \rm std }_{\infty} = q^{ \rm lin }_{\infty}=p-1/2$.
\end{proof}

\section{Examples}\label{section_examples}
\subsection{Sobolev type spaces}

Let us begin with an application of the obtained results to classes of periodic functions.
We compare the terms in Corollary \ref{cor2} in the case when
$\sigma_m \asymp_{s,d} m^{-s} (\log m)^{\alpha}$ with some $\alpha>0$. This holds, in particular, for widely used mixed Sobolev classes.

So, let $\T= [0, 2 \pi]$ be a torus where the endpoints of the interval are identified. By $\T^d$ we denote a $d$-dimensional torus
and equip it with the Lebesque measure $(2\pi)^{-d} {\rm d} \bx$.

Let further $\ltwo$ be the space of all (equivalence classes of) $2\pi$-periodic in each component measurable functions $f(\bx)= f(x_1, \dots, x_d)$
on $\T^d$ such that
$$
\|f\|_{\ltwo} = (2\pi)^{-d/2} \left( \int_{\T^d} |f(\bx)|^2 \dd \bx \right)^{1/2} < \infty.
$$
The functions $f \in \ltwo$ are completely defined by their Fourier coefficients  $c_{\bk}(f)$, $\bk\in \Z^d$, with respect to the trigonometric system
$\left\{ e^{i \bk \cdot \bx } \right\}_{\bk \in \Z^d}$
(which form an orthonormal basis in $\ltwo$):
\be\label{Fourier_coef_Ck}
 c_{\bk}(f)= (2\pi)^{-d} \int_{\T^d} f(\bx) e^{-i \bk \dot \bx} \dd \bx, \quad \bk\in \Z^d,
\ee
 namely, it holds
 $$
 f(\bx) = \sum_{\bk\in \Z^d} c_{\bk}(f)  e^{i \bk \dot \bx}.
 $$

The notation $H^w$ stands for a Hilbert space of integrable functions on the torus $\T^d$ such that
$$
\| f \|_{H^w(\T^d)} = \left( \sum_{\bk\in \Z^d} (w(\bk))^2 | c_{\bk}(f)) |^2 \right)^{1/2} < \infty,
$$
where $w(\bk)>0$, $\bk\in \Z^d$, are certain weights.

F.~Cobos, T.~K\"{u}hn and W.~Sickel~\cite[Theorem~3.1]{CoKuSi16} established necessary and sufficient condition on the weights
 $w(\bk)$, which guarantee the existence of compact embedding of $H^w(\T^d)$ in $L_{\infty}(\T^d)$, where
 $\|f\|_{L_{\infty}(\T^d)}= \operatorname{ess \, sup}_{\bx \in \T^d} |f(\bx)|$  (recall, that for the space $\ell_{\infty}(\T^d)$ we define the supremum norm).
Namely, it was proved that $H^w(\T^d) \hookrightarrow L_{\infty}(\T^d)$ compactly if and only if
\be\label{sum_1/w^2}
\sum_{\bk\in \Z^d} \frac{1}{(w(\bk))^2} < \infty \,.
\ee

Now we move to the concept of RKHS. Let us define the kernel
\be\label{kernel_w_periodic}
K_w(\bx, \by) = \sum_{ \bk\in \Z^d } \frac{ e^{i \bk \cdot (\bx - \by)}  }{ \left( w(\bk) \right)^2 }\,,
\ee
which is bounded under the condition (\ref{sum_1/w^2}).
The space $H^w(\T^d)$ is a reproducing kernel Hilbert space with the kernel (\ref{kernel_w_periodic}), and the boundedness of $K_w$
implies that $H^w(\T^d)$ is compactly embedded also into $\ell_{\infty}(\T^d)$. The eigenvalues of the operator $W_w$ are given by $\lambda_{\bk}:= (w(\bk))^{-2}$\,. Let us consider the frequency set
$$
  I(R):=\Big\{\bk \in \Z^d~:~ w(\bk) \leq R\Big\}
$$
and let $m(R)$ denote its cardinality. The reconstruction operator $S^{m(R)}_{\bX}$ is defined for the set of functions $\eta_{\bk}(\cdot) = \exp(i\bk\cdot )$ for $\bk \in I(R)$ by \eqref{algorithm_Sm}.
Note, that here we use the non-weighted least squares algorithm since
$\| \eta_k  \|_{\ell_\infty(D)} =1$,  $k\in \N$.

From Theorem~\ref{prob_l_infty_general} (Remark~\ref{remark_Non-weighted}) we get
 \begin{theorem}\label{gen_wce_l_infty_periodic}
	Let $H^w(\T^d)$ be a Sobolev type space given by weights $w(\bk)$ satisfying the condition
$$
\sum_{\bk\in \Z^d} \frac{1}{(w(\bk))^2}< \infty\,.
$$
Let further be $R\geq 1$, $r>1$ and $n$ be smallest such that
$$
	m(R) \leq \Big\lfloor \frac{n}{c_1 r\log n}\Big\rfloor \,,
 $$
 where $c_1>0$ is an absolute constant.
 Then the random reconstruction operator $S_{\bX}^{m(R)}$ satisfies
\be
\Prob \left( \sup\limits_{\|f\|_{H(K_w)}\leq 1}
\| f-  S_{\bX}^{m(R)} f\|^2_{L_{\infty}(\T^d)}
\leq C \sum\limits_{\bk:w(\bk)>R} \left( w(\bk)\right)^{-2} \right)
\geq 1- 3 n^{1-r}
\ee
with universal constant $C>0$.
\end{theorem}

In what follows, we consider concrete examples of weights $w(\bk)$, $\bk \in \Z^d$, and respective Sobolev classes of functions.
Note, that the study of different characteristics on these classes recently became of interest. It is due to the application of Sobolev classes of functions in a number of practical issues, as event modeling, quantum chemistry, signal reconstruction, etc.

\begin{theorem}\label{Th_g_m_Our_Lutz}
Under the assumptions of Theorem~\ref{gen_wce_l_infty_periodic}, the estimate
\be\label{g_m_Our_Lutz}
g_n(\operatorname{I}_w) \leq C \min\left\{ a_{  \lfloor n/ b\log  n   \rfloor} (\operatorname{I}_w)\,,
 \sqrt{\log n} \cdot a_{  \lfloor cn  \rfloor} (\operatorname{I}_w) \right\}
\ee
holds for the sampling numbers of the embedding
$\operatorname{I}_w \colon H^w(\T^d) \rightarrow L_{\infty}(\T^d)$, where $C,b,c >0$ are absolute constants. Note that
$$
  a_n (\operatorname{I}_w)^2 \leq \sum\limits_{k\geq n+1}\tau_k^2\,,
$$
where $(\tau_n)_n$ is the non-increasing rearrangement of $(1/w(\bk))_{\bk \in \Z^d}$\,.
\end{theorem}
\begin{proof}
The relation (\ref{g_m_Our_Lutz}) follows from Corollary~\ref{cor2} and the fact that $N_{w}(m) = N(m) = m-1 \leq m$. Note that
the eigenfunctions $\eta_{\bk}(\cdot)$ are always $\exp(i\bk\cdot)$ no matter what $w(\cdot)$ is.
\end{proof}

\begin{remark}
Theorem~\ref{Th_g_m_Our_Lutz} improves the recent estimate by L.~K\"{a}mmerer~\cite[Theorem~4.11]{Kam19}. There it was shown
\be\label{result_Lutz}
g_m (\operatorname{I}_w) \leq C \log m \cdot a_{  \lfloor m/ \log m  \rfloor} (\operatorname{I}_w )
\ee
under some additional conditions, that we do not need for the relation (\ref{g_m_Our_Lutz}).  Note also, that (\ref{result_Lutz})
is an improvement of the result by L.~K\"{a}mmerer and T.~Volkmer~\cite{Kam_Vol19}.
\end{remark}

\subsection{Sobolev Hilbert spaces with mixed smoothness}

Let us specify the weight sequence $(w(\bk))_{\bk}$ as follows. First, let us define $w_s^{\#}(\bk) := \prod_{j=1}^d (1+  |k_j|)^{s}$\,, $\bk\in \Z^d $,  $s > 1/2$.
By $H^{s, \#}_{\mix}(\T^d)$ we denote the periodic Sobolev space (see, e.g., \cite{KuSiUl15})
\be\label{periodic_H_sharp_norm}
 H^{s, \#}_{\mix}(\T^d) = \left\{ f\in L_2(\T^d) \colon \quad \|f\|_{H^{s, \#}_{\mix}(\T^d)} =
 \left(  \sum_{ \bk\in \Z^d }  | c_{\bk}(f) |^2  \prod_{j=1}^d (1+  |k_j|)^{2s} \right)^{1/2}  < \infty \right\}.
\ee

One could also choose the weight of the form
$w_s^{+}(\bk) := \prod_{j=1}^d \left(1+  |k_j|^2\right)^{s/2}$, $\bk\in \Z^d $, $s > 1/2$.
The corresponding classes are denoted by $H^{s, +}_{\mix}(\T^d)$, i.e.,
$$
H^{s, +}_{\mix}(\T^d) := \Bigg\{ f\in \ltwo \colon \quad
\| f \|_{ H^{s, +}_{\mix}(\T^d) } =
\Bigg( \sum_{ \bk\in \Z^d }  |  c_{\bk}(f) |^2 \prod_{j=1}^d \left(1+  |k_j|^2 \right)^s \Bigg)^{1/2}
 < \infty \Bigg\}.
$$
The norms of the classes $H^{s, \#}_{\mix}(\T^d)$ and $H^{s, +}_{\mix}(\T^d)$ are equivalent (up to constants that can depend on the dimension $d$ and the smoothness parameter $s$).
Therefore, an asymptotic behavior of sampling numbers of these classes coincide.
When considering asymptotic issues let us write $H^{s}_{\mix}(\T^d)$, which is a common notation for both introduced classes. For singular numbers of the embedding operator $\Id\colon H^{s}_{\mix}(\T^d) \rightarrow L_2(\T^d)$,
 it is known that
\be\label{sigma_star}
 \sigma_n \lesssim_{s,d} n^{-s} (\log n)^{s(d-1)} \,.
\ee

\begin{theorem}\label{asymp_Hsharp_infty}
	Let $r>1$, $m= \Big\lfloor \frac{n}{c_1 r\log n}\Big\rfloor$\,,
 where $c_1>0$ is an absolute constant.
	Then
	\be\label{asymptotic_Hsmix}
	\Prob \left( \sup_{ \| f \|_{ H^{s}_{\mix}(\T^d)} \leq 1}  \| f - S_{\bX}^m f \|_{ L_{\infty}(\T^d) }
	\lesssim_{s,d,r}    n^{-s+1/2}   \left(  \log n \right)^{sd-1/2}  \right)
	\geq 1- 3 n^{1-r}
	\ee
is true as $n \rightarrow \infty$.
\end{theorem}
\begin{proof}
The choice of $m \in \N$ yields
 $m^{-1}   \leq C_1  \frac{r \log n}{n}$\,, $C_1>0$, hence, from \eqref{sigma_star} we get
\be\label{sigma_Hsharp_asymp}
\sigma_{m } \lesssim_{s,d} r^{s} n^{-s}  \left(  \log n \right)^{sd}\,,
\ee
and therefore
\begin{align*}
	 \sum_{k\geq \lfloor m/2 \rfloor  }    \sigma_k^2
	&  \lesssim_{s,d} r^{2s}
	\sum_{k\geq \lfloor m/2 \rfloor  }  k^{-2s} (\log k)^{2s(d-1)}
 \leq  r^{2s} (m/2)^{-2s+1}
	\left( \log (m/2) \right)^{2s(d-1)}
\\
	& \lesssim_{s}  r^{2s-1}  n^{-2s+1}   \left(  \log n \right)^{2s-1}
	\left(  \log n \right)^{2s(d-1)}
 =	 r^{2s-1}  n^{-2s+1}   \left(  \log n \right)^{2sd-1}\,.
\label{sum_sigma_Hsharp_asymp}
	\end{align*}

Respectively, from Theorem~\ref{gen_wce_l_infty_periodic} we get the estimate (\ref{asymptotic_Hsmix}).
\end{proof}

\begin{theorem}\label{Stat_i_ii_sampling_per} For $H^{s}_{\mix}(\T^d)$, $s>1/2$, it holds

{\em (i)}  There is an absolute constant $b>0$ such that
$$
g_{\lfloor bm\log m \rfloor}( \Id\colon H^{s}_{\mix}(\T^d)  \rightarrow L_{\infty}(\T^d)  ) \lesssim_{s,d}
m^{-s+ 1/2} (\log m)^{s(d-1)}\,.
$$

{\em (ii)} There is an absolute constant $c>0$ such that
\be\label{asymptotic_Hsmix_improved}
g_{\lfloor cn \rfloor}( \Id\colon H^{s}_{\mix}(\T^d)) \rightarrow L_{\infty}(\T^d) ) \lesssim_{s,d}
n^{-s+ 1/2} (\log n)^{s(d-1) +1/2}\,.
\ee
\end{theorem}
\begin{proof}
The statement (i) follows immediately from Theorem~\ref{asymp_Hsharp_infty}.

To get (ii), we use instead of Theorem~\ref{gen_wce_l_infty_periodic} the improved bound for sampling numbers
from Section~\ref{section_improved}, namely, the statement (ii) of Theorem~\ref{main_samp} (Corollary~\ref{cor2}).
\end{proof}

\begin{remark}
Comparing the estimates (i) and (ii) from Theorem~\ref{Stat_i_ii_sampling_per},
 we see a reduced exponent in the logarithmic term of (ii) for  $s>1$.
\end{remark}

\begin{remark}
V.N.~Temlyakov~\cite{Te93} proved the following estimate for sparse grids:
$$
 g_n(\Id\colon H^{s}_{\mix}(\T^d) \rightarrow L_{\infty}(\T^d) ) \asymp_{s,d} n^{-s+1/2} (\log n)^{(d-1)s}\,, \quad s>1/2.
$$
The estimate of Theorem~\ref{Stat_i_ii_sampling_per} is a little worse if we look at the power of the logarithm. Let us comment on the preasymptotic estimates which give us the precise constants for sufficiently smooth function classes with $\#$-norm. Similar results can be given for the ``$+$''-norm, where we use \eqref{f30b}.

Note, that we can use the plain (non-weighted) least squares algorithm here.
 This is why we  have a slightly better $m$, $n$ scaling here.
\end{remark}

\begin{theorem}\label{preasymp_Hmix_sharp}
	 Let $s>\frac{1+\log_2 d}{2}$, $\beta:= 2s/(1+\log_2 d) >1$ and $r>1$. For $n \in \N$, $n \geq 3$, we define $m \in \N$ by
$$
	 m=\Bigl\lfloor \frac{n}{10r\log n }  \Bigr\rfloor .
$$
	Then
\be\label{th_Hmix_sharp__preasymp}
\Prob \left( \sup_{\|f   \|_{H^{s, \#}_{\mix}(\T^d)} \leq 1}  \| f - S_{\bX}^m f \|^2_{L_{\infty}(\T^d)}
\leq 1612 \left( \frac{16}{3} \right)^{\beta}  \frac{\beta}{\beta-1} \left( \frac{m}{2} -1 \right)^{-\beta +1}
\right)
\geq 1-3 n^{1-r}.
\ee	
\end{theorem}
\begin{proof}
In \cite[Theorem 4.1]{Ku19} it was established, that
	$$
	\sigma_n^{ \#} \leq \left(  \frac{16}{3n} \right)^{\frac{s}{1+\log_2 d}}\,, \quad n\geq 6.
	$$	
We obtain
	\begin{align}
	 \sum_{k\geq  \lfloor m/2 \rfloor}  \left( \sigma_k^{\#}  \right)^2
	&\leq   \left( \frac{16}{3} \right)^{\beta}
	\sum_{k \geq    \lfloor m/2 \rfloor}	k^{-\beta}
\leq   \left( \frac{16}{3} \right)^{\beta}
	\left(  \lfloor m/2 \rfloor ^{-\beta}
	+ \frac{1}{\beta-1}  \lfloor m/2 \rfloor^{-\beta +1} \right)
\nonumber    \\
    & \leq   \left( \frac{16}{3} \right)^{\beta} \frac{\beta}{\beta-1}
       \lfloor m/2 \rfloor^{-\beta+1} \leq \left( \frac{16}{3} \right)^{\beta} \frac{\beta}{\beta-1} \left( \frac{m}{2} -1 \right)^{-\beta +1}.
     \label{log_sum_per_ex1}
	\end{align}
Hence, from Theorem~\ref{prob_l_infty_general} (see Remark~\ref{remark_constants}), we have with high probability that
$$
\sup_{\|f   \|_{H^{s, \#}_{\mix}(\T^d)} \leq 1}  \| f - S_{\bX}^m f \|^2_{L_{\infty}(\T^d)} \leq
403 \max\Bigg\{ \sum\limits_{k\geq \lfloor m/2 \rfloor}\sigma_k^2, \ \
4 \sum\limits_{k\geq \lfloor m/2 \rfloor}\sigma_k^2 \Bigg\}\,,
$$
which, in combination with (\ref{log_sum_per_ex1}), yields the estimate (\ref{th_Hmix_sharp__preasymp}).
\end{proof}

\subsection{A univariate example}

\label{example_Bernardi}
Theorems~\ref{prob_l_infty_general} and \ref{main_samp} can be applied to non-bounded orthonormal systems. There is a large source of examples already for the univariate case stemming from second order differential operators in connection with orthogonal polynomials, see for instance \cite[P. 108, Lem.\ 5]{Ne79}. Here the growth of the Christoffel function for the space of polynomials of degree $n$ for different Jacobi weight parameters $\alpha,\beta>-1$ is given. It turns that if $\alpha = \beta \leq -1/2$ we have $N(n) =  \mathcal{O}(n)$ otherwise not.  We focus on Legendre polynomials ($\alpha = \beta = 0$) and the corresponding second order differential operator. Here we consider $D = [-1,1]$ together with the uniform measure $\rm d x$ on $D$. The second order operator $A$, that is defined by
$$
Af(x) = - ((1-x^2)v')' \,,
$$
characterizes for $s>1$ weighted Sobolev spaces
$$
H(K_s) := \{ f\in L_2(D)\colon \  A^{s/2}f \in L_2(D) \}\,.
$$
These spaces are RKHS with the kernel
$$
K_s(x,y) = \sum_{k\in \mathbb{N}} (1+ (k(k+1))^s )^{-1} \mathcal{P}_k(x)  \mathcal{P}_k(y)\,,
$$
where $\mathcal{P}_k\colon D \rightarrow \R$, $k \in \N$, are
 $L_2(D)$-normalized Legendre polynomials $\mathcal{P}_k(x)$.

 In this setting, we have $( \eta_k )_{k=1}^{\infty} = (\mathcal{P}_k)_{k=1}^{\infty}$\,,
$(e^*_k)_{k=1}^{\infty} = ((1+ (k(k+1))^s )^{-1/2} \mathcal{P}_k)_{k=1}^{\infty}$\,, and, accordingly,
$\sigma_k = ((1+ (k(k+1))^s )^{-1/2}$. As for the spectral function, we get
 $$
 N(m) = \sup_{x\in D} \sum_{k=1}^{m-1} |\mathcal{P}_k (x) |^2 =
 \sum_{k=0}^{m-2} \frac{2k+1}{2} = \frac{(m-1)^2}{2}\,.
 $$
 Hence,
 \begin{gather*}
 \sum\limits_{k\geq \lfloor m/2 \rfloor}\sigma_k^2 =  \sum\limits_{k\geq \lfloor m/2 \rfloor} ((1+ (k(k+1))^s )^{-1}
 \lesssim \sum\limits_{k\geq \lfloor m/2 \rfloor} k^{-2s} \lesssim_s   m^{-2s+1}
 \\
 \sum\limits_{k\geq \lfloor m/2 \rfloor}\frac{N_{K,\varrho_D}(4k)\sigma_k^2}{k}
 \lesssim \sum\limits_{k\geq \lfloor m/2 \rfloor} k k^{-2s} \lesssim_s   m^{-2s+2}\,,
\end{gather*}
\be\label{est_univariate_Th3.1}
\sup\limits_{\|f\|_{H(K_s)}\leq 1}\|f-  \widetilde{S}_{\bX}^m f\|_{L_{\infty}(D)}\lesssim_s   m^{-s+1} \lesssim n^{-s+1} (\log n)^{s-1}\,.
\ee
Applying the improved bounds for sampling numbers from Theorem~\ref{main_samp}, we derive to
the  estimate
\begin{equation}\label{eq101}
g_{n}( \Id) \lesssim_s n^{-s+1}(\log n)^{\min\{s-1,1/2\}}\,.
\end{equation}
Note, that Ch.~Bernardi and Y.~Maday~\cite[Theorem~6.2]{Bernardi_Maday1992} got optimal in the main rate estimates for approximation by a polynomial operator at Gauss points in the space $L_2(D)$. We are not aware of any existing $L_\infty(D)$ bounds in the literature. In the paper by L.~K\"ammerer, T.~Ullrich and T.~Volkmer~\citep{KUV19}, the authors obtained improved worst case error estimates with high probability in the space $L_2(D)$ (see Examples~5.7 and 5.10 there).
\begin{remark}\label{rem_bern}
Note, that using the density function $\varrho_m$ from Remark~\ref{remark_other_density},
we have the less sharp bound
\begin{equation}\label{f100}
\sup\limits_{\|f\|_{H(K_s)}\leq 1}\|f-  \widetilde{S}_{\bX}^m f\|_{L_{\infty}(D)}\lesssim_s
\sqrt{m \cdot  m^{-2s+2}} \lesssim n^{-s+3/2} (\log n)^{s-3/2} \, .
\end{equation}
Without the weighted least squares (see Remark~\ref{remark_Non-weighted}), we obtain
$$
\sup\limits_{\|f\|_{H(K_s)}\leq 1}\|f-  S_{\bX}^m f\|_{L_{\infty}(D)}\lesssim_s
m^{-s+1}
$$
that is the same in order with respect to the parameter $m$,  as those in (\ref{est_univariate_Th3.1}). But we need to impose the condition
$N(m) \sim m^2 \leq c  \, n/ \log n$, i.e.,  $m \leq C \, \sqrt{n/ \log n}$. Hence, in terms of
the number $n$ of used samples we get
\begin{equation}\label{eq100}
\sup\limits_{\|f\|_{H(K_s)}\leq 1}\|f-  S_{\bX}^m f\|_{L_{\infty}(D)}\lesssim_s
n^{-(s-1)/2} (\log n)^{ (s-1)/2 } \,,
\end{equation}
which sometimes even better than \eqref{f100}, namely if $1<s<2$. However, \eqref{f100} and \eqref{eq100} are never better than the right-hand side in \eqref{eq101}.
\end{remark}

\subsection*{Appendix. Preasymptotic estimates for singular numbers.}

Note, that the asymptotic rate of the singular numbers $\sigma_n(\Id\colon H^{s}_{\mix}(\T^d) \rightarrow L_2(\T^d))$ as $n \to \infty$
is known for a long time in many cases. So, for the periodic Sobolev classes $H^{s}_{\mix}(\T^d)$ of functions with dominating mixed smoothness the order of
approximation numbers $\sigma_n)$ was obtained by B.S.~Mityagin in 1962 (in a slightly general setting).
As for the isotropic classes $H^{s}(\T^d)$ the order of corresponding approximation numbers was proved by J.W.~Jerome in 1967. A lot of scientists were involved in finding the optimal orders of convergence for linear algorithms on different function classes, but mainly the constants in the estimates were not specified.
And in practical issues the dependence of these constants on the smoothness parameter of the spaces and the dimension of the underlying domain
is a crucial point.

 Various comments on literature concerning the existing preasymptotic results for mixed order
  Sobolev functions on the $d$-torus can be found in~\cite[Chapt.~4.5]{KuSiUl15}.
We comment only few that are closely connected to our research.

It was shown by T.~K\"{u}hn, W.~Sickel and T.~Ullrich~\cite{KuSiUl15} that for  $s>0$ and $d \in \N$ it holds
$$
\lim\limits_{n\to \infty} \frac{n^s \sigma_n (\Id\colon H^{s}_{\mix}(\T^d) \rightarrow L_2(\T^d)) }{(\log n)^{s(d-1)}} =
\left( \frac{ 2^d }{ (d-1)! } \right)^s
$$
(they showed an existence of the limit and that its value is independent of the norm choice).
The preasymptotic estimates (for small $n$) depend heavily on the choice of corresponding norm in the space $H^{s}_{\mix}(\T^d)$.
Therefore, authors further considered several natural norms ($\#, *$ and $+$) in the space $H^{s}_{\mix}(\T^d)$, in particular they showed that
for the classes $H^{s, \#}_{\mix}(\T^d)$ defined exactly as in (\ref{periodic_H_sharp_norm}) it holds for $s>0$, $d \in \N$, $d \geq 2$ and $1\leq n \leq 4^d$
the estimate
$$
\sigma_n (\Id\colon H^{s, \#}_{\mix}(\T^d) \rightarrow L_2(\T^d)) \leq \left( \frac{e^2}{n} \right)^{\frac{s}{2+ \log_2 d}}
$$
is true (with the corresponding non-matching lower bound).

This result was improved by T.~K\"{u}hn~\cite{Ku19} where he showed that for $s>0$, $d \in \N$ and all $n \geq 6$ it holds
\be\label{Kuhn_result_an}
\sigma_n (\Id\colon H^{s, \#}_{\mix}(\T^d) \rightarrow L_2(\T^d)) \leq \left( \frac{16}{3n} \right)^{\frac{s}{1+ \log_2 d}}.
\ee
The estimate (\ref{Kuhn_result_an}) contains all values of the parameter $n$, not only in preasymptotic range and moreover
the factor and the exponent in the corresponding right hand side are better.

In the case of ``+''-norm, a new preasymptotic bound for the classes of functions
with anisotropic mixed smoothness is obtained by
T.~K\"{u}hn, W.~Sickel and T.~Ullrich~\cite[Theorem~4.11]{KSU3}. In particular, it was proved that
for $s>0$, $d \geq 3$ and all $n \geq 2$ it holds
\begin{equation}\label{f30b}
\sigma_n (\Id\colon H^{s, +}_{\mix}(\T^d) \rightarrow L_2(\T^d)) \leq \left( \frac{C(d)}{n} \right)^{\frac{s}{2(1+ \log_2 (d-1))}},
\end{equation}
where the constant
$C(d) = \left( 1+\frac{1}{d-1} \left( 1+ \frac{2}{\log_2 (d-1)} \right) \right)^{d-1}$.
Note, that this estimate was stated in a more general case.
In the paper one can also find an improved upper bound in the case $2 \leq n \leq e^{d-1}$.

Note also, that  D.~Krieg~\cite[Section~4.1]{Kr18} established earlier results
on the asymptotic and preasymptotic behavior for tensor powers of arbitrary sequences
of polynomial decay, and further estimates for approximation numbers of embeddings of
 mixed order Sobolev functions on the $d$-cube (with different equivalent norms)
 into $L_2([a,b]^d)$, where $[a,b] \in \{ [0,1], [-1,1], [0, 2\pi] \}$.

As to the isotropic Sobolev classes $H^s(\T^d)$,  T.~K\"{u}hn, S.~Mayer and T.~Ullrich~\cite{KMU16} established a connection between approximation numbers of
periodic Sobolev type spaces, where the smoothness weights on the Fourier coefficients are influenced by a (quasi-)norm $\| \cdot \|$ on $\R^d$,
and entropy numbers of embeddings for finite dimensional balls. This allowed them to get preasymptotic estimates
 for isotropic Sobolev spaces and spaces of Gevrey type (in terms of equivalences where the hidden constants do not depend on $n$ and $d$).
 But in~\cite{KMU16} the authors did not indicate explicitly a dependence of these constants on the parameters $s$ and $p$.
 The exact dependence on the dimension $d$ and the smoothness $s$ was given later by T.~K\"{u}hn~\cite{Ku19}.

\paragraph{Acknowledgment.} T.U.\ would like to acknowledge support by the DFG Ul-403/2-1. The authors would like to thank Felix Bartel, David Krieg, Stefan Kunis and Mario Ullrich for valuable comments.

\end{document}